\documentclass[10pt,a4paper]{article}
\usepackage[square,sort,comma,numbers]{natbib}
\usepackage{amsthm,amssymb,amsmath}
\usepackage[pdftex]{thumbpdf}       
\usepackage{float}
\usepackage{graphicx}
\usepackage{verbatim}
\usepackage[usenames,dvipsnames]{color}
\usepackage{dsfont}
\usepackage{fullpage}
\usepackage{ifpdf}
\usepackage{comment}
\numberwithin{equation}{section}

\ifpdf
\usepackage[pdftex]{hyperref} 
\usepackage{authblk} 

\hypersetup{%
pdftitle={Fluid approximation for  EV network},
pdfauthor={ A. Aveklouris},
pdfkeywords={EV, fluid},
bookmarks=true,%
bookmarksnumbered=true,
pdfstartview={FitH},
pdfpagelayout={OneColumn},
colorlinks=true,
citecolor=NavyBlue,
linkcolor=RedOrange,
urlcolor=BrickRed,
bookmarksopen=true,%
bookmarksopenlevel=1,
unicode=true,%
breaklinks=true,%
}%
\else
  \newcommand\phantomsection\relax
  \newcommand{\url}[1]{#1}
  \newcommand{\href}[2]{#2}
  \usepackage{nohyperref}
\fi

\usepackage{url}
\mathchardef\ordinarycolon\mathcode`\: \mathcode`\:=\string"8000 \begingroup
\catcode`\:=\active \gdef:{\mathrel{\mathop\ordinarycolon}} \endgroup
 \theoremstyle{plain}              
\newtheorem{definition}{Definition}[section]
\newtheorem{remark}{Remark}[section]
\newtheorem{example}{Example}[section]
\newtheorem{theorem}{Theorem}[section]
\newtheorem{proposition}[theorem]{Proposition}

\newcommand{\ind}[1]{\mathds{1}_{\{#1\}}}   
\newcommand{\E}[1]{\mathbb{E}\left[#1\right]}  
\newcommand{\Prob}[1]{\mathbb{P}\left(#1\right)} 
\newcommand{\blt}{\boldsymbol}
\newcommand{\R}{\mathbb{R}}
\DeclareMathOperator{\Min}{\wedge}
\DeclareMathOperator{\Max}{\vee}
\graphicspath{{Figures/}}

\newcommand{\br}{\overline}
\newcommand{\wbr}{\overline}

\newcommand{\ep}{\epsilon}

\newcommand{\Node}{\mathcal{I}}
\newcommand{\Edge}{\mathcal{E}}
\newcommand{\Type}{\mathcal{J}}
\newcommand{\X}{\mathcal{X}}
\newcommand{\A}{\mathcal{A}}
\newcommand{\Q}{\mathcal{Q}}
\newcommand{\Z}{\mathcal{Z}}
\DeclareMathOperator*{\argmax}{arg\,max}

\usepackage{color}
\newcommand{\indset}{\mathds{1}}

\newcommand{\Probn}[1]{\mathbb{P}^n\left(#1\right)}

\newcommand{\res}{r}
\newcommand{\reac}{x}

\title{A Fluid Model of an Electric~Vehicle~Charging Network}
\author[1]{Angelos\ Aveklouris}
\author[2,3]{Maria\ Vlasiou}
\author[2,4]{Bert\ Zwart}
\affil[1]{The University of Chicago Booth School of Business}
\affil[2]{Eindhoven University of Technology}
\affil[3]{University of Twente}
\affil[4]{Centrum Wiskunde en Informatica}
\begin{document}

\maketitle
\begin{abstract}
We develop and analyze a measure-valued fluid model keeping track of parking and charging requirements of electric vehicles in a local distribution grid. We show how this model arises as an accumulation point of an appropriately scaled sequence of stochastic network models.
{The invariant point of the fluid model encodes the electrical characteristics of the network and the stochastic behavior of its users, and it is characterized, when it exists, by the solution of a so-called Alternating Current Optimal Power Flow (ACOPF) problem.\\}
\end{abstract}


\section{Introduction}

To deal with the effects of climate change, many countries are in the process of implementing new policy measures to stimulate the usage of electricity generated by renewable sources such as solar and wind. This comes with many societal challenges as well as opportunities for research.
The supply of energy is less predictable, which makes the task of keeping high-voltage transmission networks  reliable more challenging. In the local distribution grids, new products and services that can be used to balance the grid emerge (such as smart devices), but also create more intermittency.

In particular, electric vehicles (EVs) can cause a substantial additional load on local distribution grids \cite{2025scenario}.
This has led to a significant interest in the stochastic scheduling of electric vehicle networks. There are different ways to replenish the batteries of an EV. A stochastic network analysis of fast-charging networks has been performed in \cite{yudovina2015socially}. The analysis of a stochastic network of battery swapping infrastructures is performed in \cite{fionaswapping}.

The focus of the present paper is on analyzing congestion associated to {\em slow charging}, which happens when a car is parked while its owner is at home, at work, or shopping.  In \cite{carvalho2015critical}, it was suggested to model the evolution of slowly-charging electric vehicles in a local grid by bandwidth sharing networks, approximating the instantaneous allocation of electricity to vehicles by proportional
fairness. The main constraint that needs to be satisfied is that the voltage drop in the network needs to remain bounded.
 The focus in \cite{carvalho2015critical} was solely on simulation, assuming a Markovian model and infinitely many parking spaces for EVs.
Using simulations, the stability of proportional fairness and max-min fairness was examined.

In a recent paper \cite{aveklourisstochastic}, we proposed an extension of \cite{carvalho2015critical} by allowing for load limits, finitely many parking spaces, and deadlines (associated with parking times). The joint distribution of charging requirements and parking times was not restricted
to Markovian or independence assumptions. Using heuristic arguments, \cite{aveklourisstochastic} proposed a fluid model keeping track of the number of charged and uncharged cars in the system and an associated invariant point. This invariant point is shown to be computationally tractable in \cite{aveklourisstochastic}, as it is formulated in terms of an AC Optimal Power Flow problem with an exact convex relaxation.

The goal of the present paper is to put the analysis of \cite{aveklourisstochastic} on rigorous footing using measure-valued fluid limits.
As in \cite{aveklourisstochastic}, we allow the parking times and the charging requirements of electric vehicles to be dependent and generally distributed random variables. In addition, we consider general arrival processes with time-varying arrival rates and multiple electric vehicle types.  The distribution grid is explicitly modeled and we allow for multiple parking lots, each with finitely many parking spaces. (The fluid approximation of \cite{aveklourisstochastic} did not take the subtleties of dynamically rejecting vehicles at parking lots into consideration.)

Our work is connected to the literature on bandwidth-sharing networks. Such networks have been successfully used to model communication networks where the set of feasible schedules is determined by the maximum amount of data a communication channel can carry per time unit \cite{massoulie1999bandwidth}. The stochastic analysis of bandwidth-sharing networks was initially restricted to specific networks \cite{BonaldProutiere2002, Bonald2006}. The application of  fluid and diffusion approximations led to computationally tractable approximations of a large class of networks; see for example \cite{KangKellyLeeWilliams2008, YeYao2012, BorstEgorovaZwart2014, remerova14, ReedZwart2014, VlasiouZhangZwart2016}.

In the context of communication networks, proportional fairness is a non-trivial but justified approximation of the transmission control protocol (TCP). A similar justification in the context of EV charging is performed in
 \cite{ardakanian2013, fan2012}. In these papers, by using arguments similar to the seminal work \cite{kelly1997charging}, it is shown how algorithms like proportional fairness emerge in decentralized EV charging. Our class of controls contain proportional fairness as a special case.

Our analysis is mostly related to  \cite{remerova14}.
The main difference is that, in the setting of electric vehicles, an important constraint that needs to be satisfied is to keep the voltage drops bounded, making the bandwidth-sharing network proposed here different. This also causes new technical issues, as the capacity set can be non-polyhedral or even non-convex. In addition, arriving vehicles finding a full parking lot are discarded; we assume such cars park on a regular parking spot. This leads to the additional technical complication of a loss process in a measure-valued context.

We now describe our contributions in more detail. We develop a measure-valued fluid model for the vector process which describes the number of total and uncharged EVs in each parking lot, allowing the dynamics of the stochastic model to be approximated with a deterministic model. This model  depends on the joint distribution of the charging requirements and the parking times. We  show that our measure-valued fluid model arises as a weak limit of a vector of measure-valued processes under an appropriate scaling. Moreover, under an additional assumption on the network topology, we show that the invariant point of this dynamical system is unique and can be characterized in a computationally friendly manner by formulating a nontrivial AC optimal power flow problem (ACOPF), which is tractable as its convex relaxation is exact in many cases; as mentioned before, this characterization was observed and applied in \cite{aveklourisstochastic}.

In order to prove properties for the solutions of the fluid model, we investigate the properties of the bandwidth allocation function in our setting where the capacity set is convex. We establish similar continuity properties of the allocation function as in \cite{ReedZwart2014}.
While the structural properties of a linearized voltage model can be developed in full, for the AC power flow equations we were only able to show
continuity of the allocation function. We conjecture that Lipschitz continuity and a monotonicity property hold as well, but leave
these problems open; we refer to Sections \ref{ch5:Perturbation} and \ref{Ch5: Convergence to IP} for more specific comments.

The rest of this paper is organized as follows. In Section~\ref{Ch5:model description}, we provide
a detailed model description. In particular, we introduce our stochastic model, the power flow models that we use, and we define the system dynamics. Next, in Section~\ref{ch5:Perturbation}, we present a continuity property of the optimal power allocation. Then, we move to the stochastic model. A fluid model is presented in
Section~\ref{ch5: fluid model}, where we also study its properties.
Section~\ref{Ch5: Fluid approximation} shows that the fluid model can arise as a weak limit of the fluid-scaled processes. In Section~\ref{Ch5: Convergence to IP}, we focus on the invariant analysis of the fluid model under an additional assumption on the network topology and conclude with a counterexample of monotonicity of general tree networks.
All proofs are gathered in Sections~\ref{ch5:proofs of pertupation}--\ref{ch5:proofs conv to IP}.

\section{Model description}\label{Ch5:model description}
In this section, we provide a detailed formulation of our model and explain various notational conventions that are used in the remainder of this work.
 The model description in this section is nearly identical to that in \cite{aveklourisstochastic}. We include all details on the network structure and physical characteristic for completeness; the main difference is that the measure-valued state descriptor is fully developed and analyzed in the present paper. To this end, we also require more sophisticated notation, which we introduce first.

\subsection{Preliminaries}
We introduce the  notational conventions that will be used throughout the paper. All
vectors and matrices are denoted by bold letters. Further,  $\R$ is the set of real numbers,
$\R_+$ is the set of nonnegative real numbers, and $\mathbb{N}$ is the set of strictly positive
integers.
For two real numbers $x$ and $y$, define $x\Max y:=\max\{x,y\}$,
$x\Min y:=\min\{x,y\}$, and $x^+:=x\Max 0$. For two vectors $\blt{x}, \blt{y}\in \R^I$, define the coordinate-wise product $\blt{x}\circ\blt{y}:=(x_1y_1, \ldots, x_Iy_I)$ (i.e, the Hadamard product) and the maximum norm $\|\blt{x}\|:=\max\limits_{1\leq i \leq I}|x_i|$. Vector inequalities hold coordinate-wise, namely $\blt{x}>\blt{y}$ implies that $x_i>y_i$ for all $i$.
Furthermore, $\blt I$ represents the identity matrix and $\blt e$ and $\blt e_0$ are the vectors
consisting
of 1's and 0's, respectively, the dimensions of which are clear from the context.
Also, $\blt {e}_i$ is the vector whose $i^\text{th}$ element is 1 and the rest are all 0.

Let $Y$ be a metric space. We denote by  $\mathcal{C}(Y,Y)$ the space of continuous functions $f:Y\rightarrow Y$ and by $\mathcal{C}_b(Y,Y)$ the space of continuous and bounded functions $f:Y\rightarrow Y$.
By $\mathcal{D}(Y,Y)$ denote the space of functions $f:Y\rightarrow Y$ that are right continuous with left limits endowed with the $J_1$ topology; i.e., the Skorokhod space. Further, we write $X(\cdot):=\{X(t), t\geq 0\}$ to represent a stochastic process and $X(\infty)$ to represent a stochastic process in steady-state.  Moreover, $\overset{d}=$ and
$\overset{d}\rightarrow$ denote  equality and convergence in distribution (weak convergence).

Let $\mathcal{M}(Y)$ be the space of Randon measures (i.e., locally finite and inner regular measures) on $Y$, endowed with the Borel $\sigma$-algebra denoted by $\mathcal{B}(Y)$. Further, $\mathcal{M}_F(Y)$ is the space of the finite nonnegative measures in $\mathcal{M}(Y)$ equipped with the weak topology. We say that a sequence of measures  $\mu^n$ in
$\mathcal{M}_F(Y)$ converges to $\mu$ in the weak topology and we write
$\mu^n \overset{\mathcal{W}} \rightarrow \mu$ if and only if for each
$f\in \mathcal{C}_b(Y)$,
\begin{equation*}
\langle f,\mu_n \rangle \rightarrow
 \langle f,\mu \rangle,\ \text{as}\ n \rightarrow \infty,
\end{equation*}
where $\langle f,\mu \rangle:=\int_{Y} f(y) \mu(dy)$. Weak convergence in
$\mathcal{M}_F(Y)$  is equivalent to convergence in the Prokhorov metric: for $\mu, \nu \in \mathcal{M}_F(Y)$,
\begin{equation*}
\begin{split}
  d(\mu,\nu):=
  \inf
  \left\{\epsilon: \mu(B)\leq \nu(B^{\epsilon})+\epsilon\ \text{and}\
  \nu(B)\leq \mu(B^{\epsilon})+\epsilon \right. \\
  \left. \text{for any nonempty closed}\ B\subseteq Y
  \right\},
  \end{split}
\end{equation*}
where $B^{\epsilon}$ is the $\epsilon$-neighborhood of $B$, i.e.,
$B^{\epsilon}:=\{y\in Y: dist(y,B)\leq \epsilon\}$. When $Y=\R^k$, then $dist(y,B):=\inf\limits_{\blt{x}\in B}\|\blt{y}-\blt{x}\|$.
For $\blt{\mu}, \blt{\nu} \in \mathcal{M}_F(Y)^k$, define
\begin{equation*}
d_k(\blt{\mu},\blt{\nu})
:=\max\limits_{1\leq i\leq k}d(\mu_i,\nu_i).
\end{equation*}
It is known that $\left(\mathcal{M}_F(Y)^k, d_k\right)$ is a separate and complete space \cite{billingsley1995probability}; i.e., a Polish space. When $Y=\R^n_+$, we simplify the notation to $\mathcal{M}_F$.

\subsection{Network and infrastructure}
We consider the typical situation where a low-voltage distribution network has a tree structure. Thus, take a rooted tree $\mathcal{G}=(\Node,\Edge)$, where $\Node=\{0,1,\ldots, I\}$, denotes its set of nodes (buses) and $ \Edge$ is its set of directed edges, assuming that node $0$ is the root node (known as ``\textit{feeder}''). Denote by $\ep_{ik}\in \Edge $ the edge that connects node $i$ to node $k$, assuming that $i$ is closer to the root node than $k$. Let $\Node(k)$ and $\Edge(k)$ be the node and edge set of the subtree rooted in node $k \in \Node$. The active and reactive power consumed by the subtree $(\Node(k),\Edge(k))$ are $P_{\Node(k)}$ and $Q_{\Node(k)}$. The resistance, the reactance, and the active and reactive power losses along edge $\ep_{ik}$ are denoted by $\res_{ik}$, $\reac_{ik}$, $L^P_{ik}$, and $L^Q_{ik}$, respectively. Moreover, $V_i$ is the voltage at node $i$ and $V_0$ is known. At any node, except for the root node, there is a charging station with $K_i>0$, $i \in \Node \setminus \{0\}$, parking spaces (each having an EV charger). Further, we assume that there are $\Type=\{1,\ldots,J\}$ different types of EVs indexed by $j$.

\subsection{Stochastic model for EVs}
Type-$j$ EVs arrive at node $i$ according to a counting process $E_{ij}(\cdot):=\{E_{ij}(t), t\geq 0\}$; i.e., $E_{ij}(t)$ is the number of EVs that arrive into the parking lot in the time interval $(0,t]$. We assume that all $E_{ij}(\cdot)$ are finite, nondecreasing processes with $E_{ij}(0)=0$,  $E_{ij}(t)-E_{ij}(t^-)\in \{0,1\}$, and
$\E{E_{ij}(t)}=\int_{0}^{t} \lambda_{ij}(s)ds$ where $\lambda_{ij}(s)>0$ are integrable functions.
Moreover, let $\zeta_{ijl}$ denote the arrival time of the $l^{\text{th}}$ type-$j$ EV at node $i$.
If all spaces are occupied, a newly arriving EV does not enter the system, but is assumed to leave immediately.

We now turn to the model characteristics of the EVs. Each EV has a random charging requirement (counted in time) and a random parking time. These depend on the type of the EV and the location that it is parked (i.e., the node), but are independent between EVs. Our framework is general enough to distinguish between types. For example, we can classify types according to intervals of ratio of the charging requirement and parking time and/or according to the contract they have with the network provider. An EV leaves the system after its parking time expires. It may not be fully charged. If an EV finishes its charge, it remains at its parking space without consuming power until its parking time expires. EVs that have finished their charge are called ``\textit{fully charged}''.

Let $B_{ijl}$ and $D_{ijl}$ denote the charging requirement and the parking time of the $l^{\text{th}}$ EV of type-$j$ at node $i$.  In queueing terminology, these  quantities are respectively called \textit{service requirements} and \textit{deadlines}. Moreover, we assume that the sequence $\left\{B_{ijl}, D_{ijl},\ l \in \mathbb{N}\right\}$ is a sequence of i.i.d.\ copies of a random vector $\left(B_{ij}, D_{ij}\right)$  with  distribution law $F_{ij}(A) = \Prob{(B_{ij}, D_{ij})\in A}$ for any Borel set $A\in \mathcal{B}(\R^2_+)$. Further, for $l=1, \ldots, Q_{ij}(0)$ we denote by $(B_{ijl}^0, D_{ijl}^0)$  the residual charging requirement and the residual parking time of the initial population of type-$j$ at node $i$. Moreover, we assume the probability density function (pdf) of the parking times $f_{D_{ij}}(\cdot)$ exists  with $f_{D_{ij}}(0)>0$ for any $i,j\geq1$. Note that it is possible to communicate an indication of the charging time $B$ and the parking time $D$ by the owner of an EV \cite{arif2016}.
The model is illustrated in Figure~\ref{fig:model}.
\begin{figure}[t]
\centering
\includegraphics[scale=0.5]{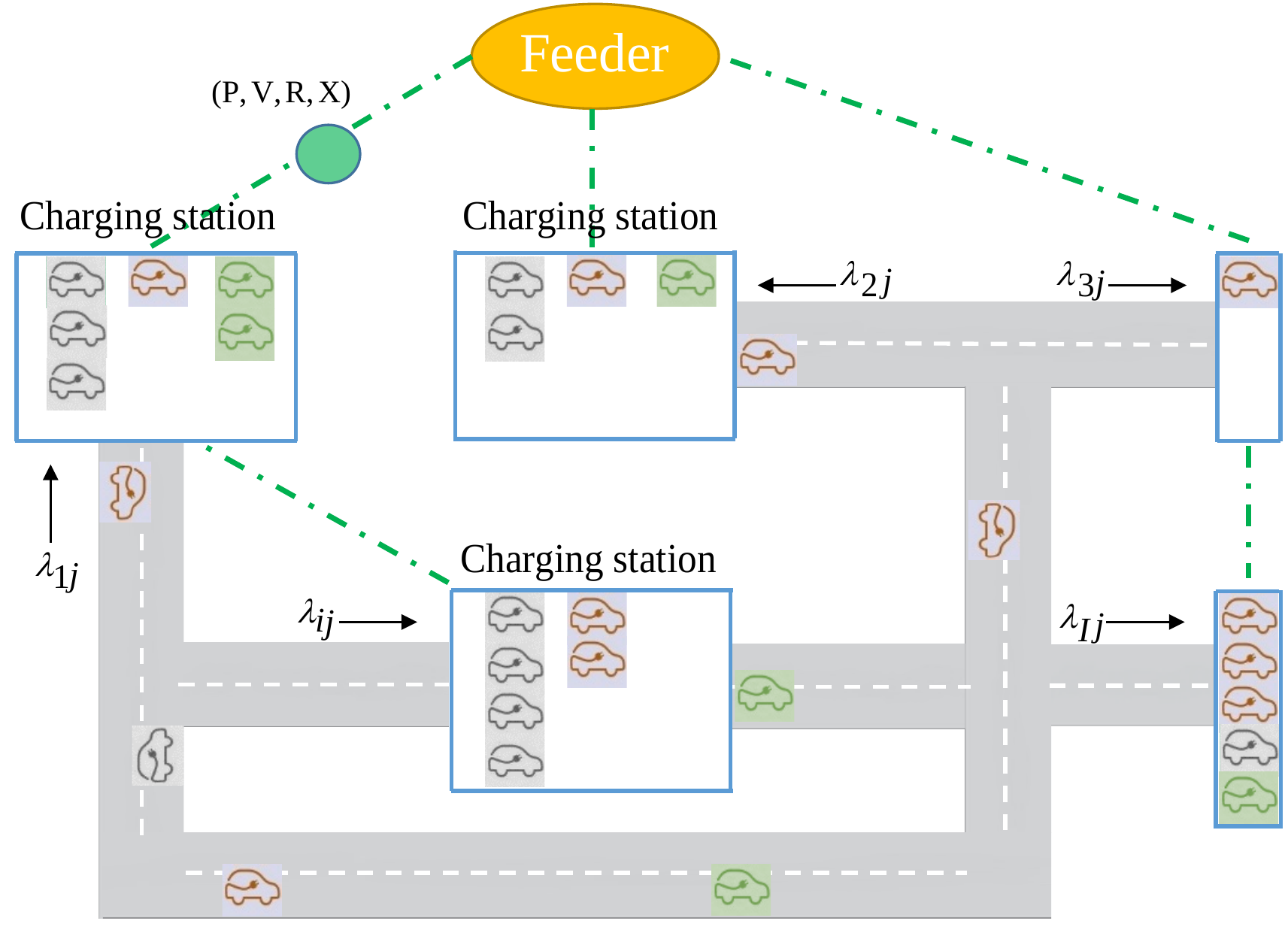}
\caption{A network with $\Type$ types of EVs and constant arrival rates.}
\label{fig:model}
\end{figure}

\subsection{Charging control rule}\label{sec:charging control rule}
An important part of our framework is the way we specify how the charging of EVs takes place. Let the number of uncharged vehicles (of all types and in all nodes) be given by the vector $\blt{z}\in [0,\infty)^{I \times J}$; i.e., $z_{ij}$ is the number of uncharged vehicles of type-$j$ in node $i$. We assume the existence of a vector function $\blt p(\blt z)=(p_{ij}(\blt z): i\in \Node \setminus \{0\},\ j\in \Type)$ that specifies the instantaneous rate of power each uncharged vehicle receives.
Moreover, we assume that this function is obtained by optimizing a ``global'' function. Specifically, for a type-$j$ EV at node $i$,  we associate a function $u_{ij}(\cdot)$ which is strictly increasing and concave in $\R_{+}$, twice differentiable in $(0,\infty)$ with $\lim_{x \rightarrow 0} u_{ij}'(x)=\infty$. The  charging rate $\blt p(\blt z)$ is then determined by
$\underset{}{\text{max}_{\blt p}}  \sum_{i=1}^{I} \sum_{j=1}^{J} z_{ij} u_{ij}(p_{ij})$
subject to a number of constraints that take into account physical limits on the charging of the batteries, load limits, and most importantly voltage drop constraints.
An important example is the choice $u_{ij}(p_{ij}) = w_{ij} \log p_{ij}$, which is known as \textit{weighted proportional fairness}.

We next introduce the physical constraints of the network. The maximum electric power that can be consumed in total by all cars
is $M_i>0$ at node $i$.  Each type-$j$ EV can be charged at a rate that is at most equal to  $c_{j}^{\text{max}}$. That is,
 \begin{equation}\label{eq:nodecon}
    \sum_{j=1}^{J} z_{ij}p_{ij} \leq M_i \quad \text{and}
    \quad
  0 \leq p_{ij} \leq c^{max}_j.
 \end{equation}
We refer to \eqref{eq:nodecon} as ``\textit{load constraints}''.
In addition, we impose \textit{voltage drop constraints}. These constraints rely on the load flow model used. Two of these models that we consider are described next.
\subsubsection{A simplified AC voltage model}
We consider a simplification of the full AC power flow equations, based on the typical situation that voltage angle differences in distribution networks are negligible
\cite[Chapter 3]{kersting2012distribution}. Under this assumption, Kirchhoff's law \cite[Eq.~1]{low14I} takes the form, for $\ep_{pk}\in \Edge$,
\begin{equation}\label{eq:AC}
  V_p V_k-V_kV_k-P_{\Node(k)} \res_{pk}-Q_{\Node(k)} \reac_{pk}=0,
\end{equation}
where $p\in \Node$ denotes the unique parent of node $k$.
The previous equations are non-linear. Applying the transformation
\begin{align*}
\blt{W}(\ep_{pk})
=\left(
  \begin{array}{cc}
    V_p^2 & V_p V_k \\
    V_k V_p & V_k^2 \\
  \end{array}
\right)
=:
\left(
  \begin{array}{cc}
    W_{pp} & W_{pk} \\
    W_{kp} & W_{kk} \\
  \end{array}
\right)
\end{align*}
leads to linear equations (in terms of $\blt{W}(\ep_{pk})$),
\begin{equation}\label{eq:KVL}
W_{pk}-W_{kk}-P_{\Node(k)} \res_{pk}-Q_{\Node(k)} \reac_{pk}=0,\ \ep_{pk}\in \Edge.
\end{equation}
Note that $\blt{W}(\ep_{pk})$ are positive semidefinite matrices (denoted by $\blt{W}(\ep_{pk})  \succeq  0$) of rank one.
The active and reactive power consumed by the subtree $(\Node(k),\Edge(k))$ are given by
\begin{align}
P_{\Node(k)}&= \sum_{l\in \Node(k)}\sum_{j=1}^{J}
z_{lj} p_{lj} +\sum_{l\in \Node(k)}
\sum_{\ep_{ls}\in \Edge(k)}L^P_{ls}, \label{eq:power}\\
Q_{\Node(k)}&=\sum_{l\in \Node(k)}\sum_{\ep_{ls}\in \Edge(k)}L^Q_{ls},\nonumber
\end{align}
where
by \cite[Appendix~B]{carvalho2015critical},
\begin{align*}\label{eq:APL}
L^P_{ls}&=(W_{ll}-2W_{ls}+W_{ss})\res_{ls}/(\res^2_{ls}+\reac^2_{ls}),\\
L^Q_{ls}&=(W_{ll}-2W_{ls}+W_{ss})\reac_{ls}/(\res^2_{ls}+\reac^2_{ls}).
\end{align*}
Note that $W_{kk}$ are dependent on the vectors $\blt p$ and $\blt z$. We sometimes write $W_{kk}(\blt p,\blt z)$ when we wish to emphasize the dependence. The function
$\blt p(\blt z)$ is given by
\begin{equation}\label{OP}
\begin{aligned}
& \underset{}{\max\limits_{\blt{p},\blt{W}}}
& & \sum_{i=1}^{I} \sum_{j=1}^{J} z_{ij} u_{ij}(p_{ij}) \\
& \text{subject to}
&&   \eqref{eq:nodecon},\eqref{eq:KVL},
\ \underline{ \upsilon}_i \leq W_{ii} \leq \overline{\upsilon}_i ,\\
&&& \blt{W}(\ep_{ik})  \succeq  0,\ \text{rank}(\blt{W}(\ep_{ik}))=1,\ \ep_{ik}\in \Edge,
\end{aligned}
\end{equation}
for $z_{ij}>0$. If $z_{ij}=0$, then  $p_{ij}=0$. In addition, $0<\underline{ \upsilon}_k \leq W_{00} \leq \overline{\upsilon}_k $ are the voltage limits.
Observe that the optimization problem~\eqref{OP} is non-convex and in general NP hard due to the rank-one constraints. Removing the non-convex constraints yields a convex relaxation, which is a second-order cone program, namely
\begin{equation}\label{ROP}
\begin{aligned}
& \underset{}{\max\limits_{\blt{p},\blt{W}}}
& & \sum_{i=1}^{I} \sum_{j=1}^{J} z_{ij} u_{ij}(p_{ij}) \\
& \text{subject to}
&&   \eqref{eq:nodecon},\eqref{eq:KVL},
\ \underline{ \upsilon}_i \leq W_{ii} \leq \wbr{\upsilon}_i ,\\
&&& \blt{W}(\epsilon_{ik})  \succeq  0,\ \epsilon_{ik}\in \Edge.
\end{aligned}
\end{equation}
Further, by Remark~\ref{re:AlCon} (see below) and
 \cite[Theorem 5]{low14}, we obtain that the convex-relaxation problem is exact. Defining
the bandwidth allocation function
$\blt{\Lambda}(\blt{z}):=\blt{p}(\blt{z})\circ\blt{z}$, i.e.,
$\Lambda_{ij}(\blt{z})=p_{ij}(\blt{z})z_{ij}$ for $i,j\geq 1$, the optimization problem (OP) \eqref{ROP} takes the following equivalent  form
\begin{equation}\label{ROPband}
\begin{aligned}
& \underset{}{\max\limits_{\blt{\Lambda},\blt{W}}}
& & \sum_{i=1}^{I} \sum_{j=1}^{J} z_{ij} u_{ij}\left(\Lambda_{ij}/z_{ij}\right) \\
& \text{subject to}
&&   \sum_{j=1}^{J} \Lambda_{ij} \leq M_i, \quad
  0 \leq \Lambda_{ij} \leq z_{ij}c^{max}_j, \\
&&& \eqref{eq:KVL}, \ \underline{ \upsilon}_i \leq W_{ii}(\blt{\Lambda}) \leq \wbr{\upsilon}_i ,\\
&&& \blt{W}(\epsilon_{ik})  \succeq  0,\ \epsilon_{ik}\in \Edge.
\end{aligned}
\end{equation}
Note that the constraints $\blt{W}(\epsilon_{ik})  \succeq  0$ are equivalent
to $W_{ii}W_{kk}-W_{ik}^2\geq 0$, since we consider $W_{ii}>0$ for any node $i\geq 1$. In the sequel, we freely use both formulations.
\subsubsection{Linearized Distflow model}\label{sec:Distflow}
Though the previous voltage model is tractable enough for a convex relaxation to be exact, it is rather complicated. Assuming that
the active and reactive power losses on edges are small relative to the power flows, but now allowing the voltages to be complex numbers, we arrive at a linear approximation of the previous model, called the \textit{linearized (or simplified) Distflow model} \cite{Baran89}.
In this case, the voltage magnitudes $W_{kk}^{lin}:=|V_{k}^{lin}|^2$ have an analytic expression \cite[Lemma~12]{low14I}:
\begin{equation}\label{eq:Dist}
  W_{kk}^{lin}(\blt p,\blt z)=W_{00}
  -2\sum_{\ep_{ls}\in \mathcal{P}(k)}\res_{ls}\sum_{m \in \Node(s)}
  \sum_{j=1}^{J} z_{mj}p_{mj},
\end{equation}
where the $\mathcal{P}(k)$ is the unique path from the feeder to node $k$.
\begin{remark}\label{re:AlCon}
Note that $W_{kk}^{lin}\leq W_{00}$ for all nodes $k$, as we assume that the nodes only consume power, and by \cite[Lemma~12]{low14I} we obtain
$W_{kk}(\blt p,\blt z) \leq W_{kk}^{lin}(\blt p,\blt z)$. That is, we can remove the constraints $W_{kk}(\blt{\Lambda}) \leq \wbr{\upsilon}_k $ from \eqref{OP}.
\end{remark}
To derive the representation of the power allocation mechanism $\blt p (\blt z)$ in this setting, one replaces the constraints in~\eqref{OP} by \eqref{eq:nodecon} and
$\underline{ \upsilon}_k \leq W^{lin}_{kk}(\blt p,\blt z)$.

\subsection{State descriptor}\label{ch5:State descriptor}
In this section, we introduce the dynamics that describe the evolution of the system. Specifically, we will now incorporate in the system dynamics all residual processes needed to obtain a Markovian system.
Let $\mathcal{Q}_{ij}(\cdot)$ and $\mathcal{Z}_{ij}(\cdot)$ be non-negative discrete  measures for $i,j\geq 1$. The total number of type-$j$ EVs at node $i$ at time $t>0$ and the number of uncharged EVs are given by
$Q_{ij}(t)=\langle 1,\mathcal{Q}_{ij}(t)\rangle$
and
$Z_{ij}(t)=\langle 1,\mathcal{Z}_{ij}(t)\rangle$,
respectively.
Moreover, $ Q_i(t):=\sum_{j=1}^{J} Q_{ij}(t)$
gives the total number of EVs at node $i\geq 1$.

Recall that $\zeta_{ijl}$ is the arrival time of the $l^{\text{th}}$ EV of type-$j$ at node $i$.
The residual parking time of the $l^{\text{th}}$ newly arriving EV of type-$j$  can be written as
$ D_{ijl}(t):=\left(D_{ijl}-(t-\zeta_{ijl})\right)^+ $, $l=1,\ldots, E_{ij}(t)$
and for the initial population
$ D_{ijl}^0(t):=(D_{ijl}^0-t)^+ $, $l=1,\ldots, Q_{ij}(0)$.
In order to define the residual charging requirements, we first introduce the following operators:
\begin{equation}\label{eq:service}
S_{ij}(\blt{z},s,t)=\int_{s}^{t} p_{ij}(\blt{z}(u)) du.
\end{equation}
For $s\leq t$, $S_{ij}(\blt{Z},s,t)$ is the \emph{cumulative bandwidth} allocated per type-$j$ EV at node $i$ during time interval $[s,t]$. The residual charging requirement of the $l^{\text{th}}$ type-$j$ EV at node $i$ at time $t\geq 0$ is given by
\begin{equation*}\label{eq:band}
B_{ijl}(t)=\left(B_{ijl}-S_{ij}(\blt{Z},\zeta_{ijl},t) \right)^+,
\end{equation*}
for the newly arriving EVs  and
$B_{ijl}^0(t)=(B_{ijl}^0-S_{ij}(\blt{Z},0,t) )^+$, $l=1,\ldots, Z_{ij}(0)$, for the initially uncharged EVs. Now, we define the measure-valued state descriptor  for any $t\geq 0$ and for any Borel set
$B \subseteq \R$,
\begin{equation}\label{eq:MVQ}
\mathcal{Q}_{ij}(t)(B):=\sum_{l=1}^{Q_{ij}(0)}
\delta^{+}_{D_{ijl}^0(t) }\left(B\right)
+\sum_{l=1}^{E_{ij}(t)} \delta^{+}_{ D_{ijl}(t) }\left(B\right) \ind{ Q_i(\zeta_{ijl}^-)<K_i}.
\end{equation}
The measure $\delta^{+}_{\cdot}(B)$ is the Dirac measure restricted on $(0,\infty)$; i.e., $\delta^{+}_{x}(B):=\delta_{x}(B\cap(0,\infty))$ and $\delta_{x}(B)=1$ if $x\in B$. The measure $\mathcal{Q}_{ij}(t)(B)$ counts the total number of type-$j$ EVs in node $i$ whose residual parking time belongs to the Borel set $B$.

The number of uncharged EVs for which the minimum between their residual charging requirement and their residual parking time belongs to any Borel set $B' \subseteq \R^2$ is given by
\begin{equation}\label{eq:MVZ}
\mathcal{Z}_{ij}(t)(B'):=\sum_{l=1}^{Z_{ij}(0)}
\delta^{+}_{\left(B_{ijl}^0(t), D_{ijl}^0(t) \right)} \left(B'\right)
+\sum_{l=1}^{E_{ij}(t)} \delta^{+}_{\left(B_{ijl}(t), D_{ijl}(t) \right)}\left(B'\right)
\ind{ Q_i(\zeta_{ijl}^-)<K_i}.
\end{equation}
The measure $\delta^{+}_{(\cdot, \cdot)}(B')$ is the Dirac measure restricted on $(0,\infty)^2$; i.e., $\delta^{+}_{\left(x_1,x_2\right)}(B'):=
\delta_{\left(x_1,x_2\right)}(B'\cap(0,\infty)^2)$ and
 $\delta_{\left(x_1,x_2\right)}(B')=1$ if $x_1\Min x_2\in B'$. Last, note that
 $\{ Q_i(\zeta_{ijl}^-)<K_i\}$ represents the event that there is an idle EV charger right before the arrival of the $l^{\text{th}}$ type-$j$ EV.
As not all EVs enter the system, we naturally define the following stochastic processes.
First, the number of accepted type-$j$ EVs at node $i$ until time $t>0$ is given by
\begin{equation}\label{eq:Accepted}
A_{ij}(t)= \sum_{l=1}^{E_{ij}(t)}
 \ind{Q_i(\zeta_{ijl}^-) <K}.
\end{equation}
Next, the number of rejected EVs until time $t>0$ is given by
\begin{equation}\label{eq:Rejected}
R_{ij}(t)=\sum_{l=1}^{E_{ij}(t)}
 \ind{Q_i(\zeta_{ijl}^-) =K}.
\end{equation}
Observe that the following relation holds:
$  A_{ij}(t)+R_{ij}(t)=E_{ij}(t)$.

Having introduced the stochastic model, which is defined through equations \eqref{eq:MVQ}--\eqref{eq:Rejected}, we move to the main results of this paper. We first study some properties of the bandwidth allocation function in Section~\ref{ch5:Perturbation}. We then  define an appropriate fluid model in Section~\ref{ch5: fluid model} and derive some of its properties.

\section{Continuity of the optimal allocation function}\label{ch5:Perturbation}

In this short section, we state some structural properties of the optimal allocation function, which may be of independent interest. In particular,  we show that the optimal solution of \eqref{ROPband} is continuous under the AC power flow model \eqref{eq:KVL}. This result is needed in Section~\ref{Ch5: Fluid approximation} in order to show convergence of the fluid-scaled processes. Last, in power system analysis, rigorous proofs are typically difficult and require additional assumptions on the distribution system \cite{dvijotham2017high}, even if one ignores the stochastic dynamics. In the rest of this section, we make the additional assumption that the ratio of resistance and reactance is constant for all the lines, i.e., $\frac{\res_{pl}}{\reac_{pl}}$ remains constant for any $\epsilon_{pk}\in \Edge$.

We show that the optimal aggregated  power allocation $\blt{\Lambda}(\blt{z})$, $\blt{z}\in (0,\infty)^{I \times J}$, is a continuous function in $\blt{z}$.  In order to establish this property, we first present a preliminary result.
\begin{proposition}\label{prop:feasible}
Let $\blt{z}\in [0,\infty)^{I \times J}$ and $\blt{\Lambda}(\blt{z})$ be a feasible point of \eqref{ROPband}.  Given a point
$\blt{0}\leq \blt{\Lambda}' \leq \blt{\Lambda}(\blt{z})$, we have that $\blt{\Lambda}'$ is also a feasible point of \eqref{ROPband}.
\end{proposition}
Observe that in case the feasible set of \eqref{ROPband} is polyhedral, the conclusion is immediate. The proof of the previous proposition is given in Section~\ref{ch5:proofs of pertupation}. The main idea of the proof is to construct a new solution $(\blt{W}', \blt{\Lambda}')$. Then, using the feasibility of the point $\blt{\Lambda}(\blt{z})$ and induction starting from the leaf nodes, we show that the point $(\blt{W}', \blt{\Lambda}')$ lies in the feasible set of \eqref{ROPband}.
In the sequel, we present the main result of this section, which says that $\blt{\Lambda}(\cdot)$ is a continuous function.
\begin{theorem}[Continuity]\label{Ch5Pr:continuous}
Let $\blt{\Lambda}(\blt{z})$ for $\blt{z}\in (0,\infty)^{I \times J}$ be the unique optimal solution of \eqref{ROPband}.
We have that $\blt{\Lambda}(\blt{z})$ is a continuous function in
$(0,\infty)^{I \times J}$.
\end{theorem}
The proof of Theorem~\ref{Ch5Pr:continuous} is given in Section~\ref{ch5:proofs of pertupation} and it combines Proposition~\ref{prop:feasible}, the continuity property of the voltages as functions of loads, and arguments from \cite[Lemma~1]{remerova14}.

In Section~\ref{ch5: fluid model}, and more specifically when we prove that the fluid model solution is unique, we need the stronger property that the optimal solution of \eqref{ROPband} is Lipschitz continuous. If we assume the linearized Distflow power model (see Section~\ref{sec:charging control rule}), then the feasible set of \eqref{ROPband} is polyhedral, and hence $\blt{\Lambda}(\cdot)$ is Lipschitz continuous by applying directly \cite[Theorems~3.1 and 3.2]{ReedZwart2014}.
In the case of the AC power flow model, where the feasible set of \eqref{ROPband} is convex,  we need to make an additional assumption that the strict complimentary condition holds for some constraints before we can conclude Lipschitz continuity.
While we have not been able to establish this property without the aforementioned assumption,  we conjecture that $\blt{\Lambda}(\cdot)$ is Lipschitz continuous, and leave this question open.

We now move to the original stochastic network and its fluid model.

\section{Fluid model definition}\label{ch5: fluid model}
In this section, we define and study the properties of a deterministic fluid model, associated with the stochastic model introduced in Section~\ref{Ch5:model description}. All proofs of this section are gathered in Section~\ref{ch5:proofof Unique}.

Define the following classes
$$\mathcal{C}:=\left\{[x,\infty),\ x\in \R_+ \right\}$$ and $$\mathcal{C'}:=\left\{[x,\infty)\times[y,\infty),\ x,y\in \R_+ \right\}.$$
Further, for any
$A\in \mathcal{C}$ and  $s\in \R$, define $A+s:=\left\{s+y,\ y\in A\right\}$ and for any $A\in \mathcal{C'}$ and $(s,t)\in \R^2$,
define $A'+(s,t):=\left\{s+[x,\infty)\times t+[y,\infty),\ x,y\in A'\right\}$.

\begin{definition}[Fluid model]\label{def:FMSC}
Let the initial data for the fluid model be given by
$$\left(\wbr{\blt{E}}(\cdot),
\wbr{\blt{\Q}}(0), \wbr{\blt{\Z}}(0)\right) \in
 C(\R_+,\R_+)\times\mathcal{M}_F^{I\times J}\times\mathcal{M}_F^{I\times J},$$ where
$\wbr{E}_{ij}(t)=\int_{0}^{t}\lambda_{ij}(s)ds$.
We say that the vector
$$\left(\blt{\wbr{\Q}}(\cdot),
\blt{\wbr{\Z}}(\cdot),\blt{\wbr Q}(\cdot), \blt{\wbr Z}(\cdot)\right)
\in
 C(\R_+,\mathcal{M}_F^{I\times J})^2\times C(\R_+,R_+^{I\times J})^2$$
is a fluid model solution
if $\wbr Q_{ij}(t)=\langle 1,\wbr{\Q}_{ij}(t)\rangle$,
$\wbr Z_{ij}(t)=\langle 1,\wbr{\Z}_{ij}(t)\rangle$,
and if there exist nondecreasing nonnegative continuous functions $\wbr{R}_{i}(\cdot)$,
$\wbr{R}_{ij}(\cdot)$ such that
\begin{align*}
\wbr{R}_{i}(t)=\int_{0}^{t}
\ind{\wbr Q_{i}(s) =K_i}d\wbr{R}_{i}(s)\hspace{0.2cm}
\text{and} \hspace{0.2cm}
\wbr{R}_{ij}(t)=\int_{0}^{t} \frac{\lambda_{ij}(s)}{\sum_{h=1}^{J}\lambda_{ih}(s)}d\wbr{R}_{i}(s).
\end{align*}
Furthermore,  for any $t\geq 0$, $\mathcal{A} \in \mathcal{C}$, and $\mathcal{A'} \in \mathcal{C'}$ the following relations hold
\begin{equation}\label{ch5eq:FLQ}
\begin{split}
\wbr{\Q}_{ij}(t)(A)= \wbr Q_{ij}(0) \Prob{ D_{ij}^0\in A+t }+
\int_{0}^{t}  \Prob{ D_{ij} \in A+(t-s) } d\wbr{E}_{ij}(s)\\
-\int_{0}^{t}  \Prob{ D_{ij} \in A+(t-s) } d\wbr{R}_{ij}(s),
\end{split}
\end{equation}
\begin{equation*}
\begin{split}
\wbr{\Z}_{ij}(t)(A')= \wbr Z_{ij}(0)& \Prob{ (B_{ij}^0,D_{ij}^0) \in A'+(S_{ij}(\blt{z},0,t),t) }\\
&+
\int_{0}^{t}  \Prob{ (B_{ij},D_{ij}) \in A'+(S_{ij}(\wbr{\blt{Z}},s,t),t-s) } d\wbr{E}_{ij}(s)\\
&-\int_{0}^{t}  \Prob{(B_{ij},D_{ij}) \in A'+(S_{ij}(\wbr{\blt{Z}},s,t),t-s) } d\wbr{R}_{ij}(s).
\end{split}
\end{equation*}
Moreover, the functions
$\wbr Q_{ij}(\cdot)=\langle 1,\wbr{\Q}_{ij}(\cdot)\rangle=\wbr{\Q}_{ij}(\cdot)(\R_+)$ and $\wbr Z_{ij}(\cdot)=\langle 1,\wbr{\Z}_{ij}(\cdot)\rangle=\wbr{\Z}_{ij}(\cdot)(\R_+^2)$ are given by
\begin{align}\label{eq:qnet}
\begin{split}
\wbr Q_{ij}(t)=
\wbr Q_{ij}(0) \Prob{D_{ij}^0 \geq t }+
\int_{0}^{t}  \Prob{D_{ij} \geq t-s } d\wbr{E}_{ij}(s)\\
-
\int_{0}^{t}  \Prob{ D_{ij} \geq t-s } d\wbr{R}_{ij}(s)
\end{split}
\end{align}
and
\begin{equation*}
\begin{split}
\wbr Z_{ij}(t)=
\wbr Z_{ij}(0)& \Prob{B_{ij}^0 \geq S_{ij}(\wbr{\blt{Z}},0,t)  ,D_{ij}^0 \geq t }\\
&+
\int_{0}^{t}  \Prob{B_{ij} \geq S_{ij}(\wbr{\blt{Z}},s,t)  ,D_{ij} \geq t-s } d\wbr{E}_{ij}(s)\\
&-
\int_{0}^{t}  \Prob{ B_{ij} \geq S_{ij}(\wbr{\blt{Z}},s,t)  ,D \geq t-s } d\wbr{R}_{ij}(s).
\end{split}
\end{equation*}
\end{definition}
We call the vectors $(\blt{\wbr{\Q}}(\cdot)$, $\blt{\wbr{\Z}}(\cdot))$ and $(\blt{\wbr Q}(\cdot)$, $\blt{\wbr Z}(\cdot))$ the
\emph{measure-valued fluid model solution} and the
 \emph{numeric fluid model solution}, respectively.

The fluid model equations, though still rather complicated, have an intuitive meaning. For instance, the term
$\Prob{B_{ij} \geq S_{ij}(\wbr{\blt{Z}},s,t)  ,D_{ij} \geq t-s }$
resembles the fraction of EVs of type-$j$ admitted to the system at time $s$
at node $i$ that are still in the system at time $t$. For this to
happen, their deadline needs to exceed $t-s$ and their service
requirement needs to be bigger than the service allocated,
which is $S_{ij}(\wbr{\blt{Z}},s,t)$. In addition, $\wbr{R}_{ij}(t)$ represents the lost fluid of type-$j$ EVs at node $i$ due to a full system until time $t\geq0$.

\begin{remark}\label{ch5:remark1}
Note that the sets $\mathcal{C}$ and $\mathcal{C'}$ generate the Borel $\sigma-$algebra of $\R$ and $\R^2$, respectively. Then, by  Dynkin's $\pi$-$\lambda$ theorem, the fluid model solutions hold for any Borel set. See Section~2.3 in \cite{gromoll2008} for more details. Moreover, by \cite[Remark~3.2]{remerova14}, fluid model solutions are invariant with respect to time shifts.
\end{remark}

We next show that  the total number of EVs in the fluid model can be rewritten in a familiar form for queueing systems and the departure process in the fluid model can be written as a function of the total number of EVs.
\begin{proposition}\label{ch5:analysisQ}
We have that for any $i\geq 1$ and $j\geq 1$,
\begin{equation}\label{eq:Departures}
\wbr{Q}_{ij}(t)=
 \wbr{Q}_{ij}(0) +\wbr{E}_{ij}(t)-\wbr{R}_{ij}(t) -\wbr{D}_{ij}(t),
\end{equation}
where $\wbr{D}_{ij}(t)$ represents the amount of fluid that departs from the system in time interval $[0,t)$, and
\begin{equation}\label{eq:depE}
\wbr{D}_{ij}(t)=\int_{0}^{t}\lim_{\epsilon\rightarrow 0}
\frac{\wbr{Q}_{ij}(s)-\wbr{\Q}_{ij}(s)\left([\epsilon,\infty)\right)}
{\epsilon}ds <  \infty.
\end{equation}
\end{proposition}
The last proposition uses the assumption of existence of the density of the parking times in order to ensure that the limit in \eqref{eq:depE} exists, and this is  the only point where we need this assumption. It follows from Proposition~\ref{ch5:analysisQ} that the total number of EVs can be written with the help of a one-dimensional reflection mapping. This result will be helpful, when we show uniqueness of the fluid model solution in Theorem~\ref{Ch5thm:Uniqueness}.
The novelty in our setting  is \eqref{eq:depE}, where an intuitive explanation is as follows.  The difference $\wbr{Q}_{ij}(s)-\wbr{\Q}_{ij}(s)\left([\epsilon,\infty)\right)$ represents the amount of fluid of type-$j$ EVs at node $i$ for which its residual parking time lies in the interval $(0,\epsilon)$. It is natural now to expect that by dividing the last difference  by $\epsilon>0$ and by allowing $\epsilon$ to be arbitrary small, the quantity $\lim\limits_{\epsilon\rightarrow 0}
\frac{\wbr{Q}_{ij}(s)-\wbr{\Q}_{ij}(s)\left([\epsilon,\infty)\right)}
{\epsilon}$ represents the departure rate of an EV from the parking lot at time $s>0$. Observe that \eqref{eq:depE} corresponds to \cite[Equation~3.2]{kang2015}. However, in the latter, the authors use different test functions to define the fluid model and they write the departure rate in terms of the hazard rate function.

Before we continue our analysis, we present an example in case of a Markovian model.
\begin{example}[Markovian model]
Consider a Markovian model (i.e., Poisson arrival process with constant arrival rate and exponential parking times), and  take $J=1$ and $\wbr{Q}_i(0)=0$ for convenience.
We shall show that the departure process given in \eqref{eq:depE} can be written in the well-known form for a Markovian model \cite{pang2007martingale}, namely
\begin{equation}\label{ch5:markov}
\wbr{D}_{i}(t)=\int_{0}^{t}\lim_{\epsilon\rightarrow 0}
\frac{\wbr{Q}_{i}(s)-\wbr{\Q}_{i}(s)\left([\epsilon,\infty)\right)}
{\epsilon}ds
=\frac{1}{\E{D_{i}}}\int_{0}^{t}
\wbr{Q}_i(s)ds.
\end{equation}
To show \eqref{ch5:markov}, use 
\eqref{eq:qnet} and $A=[\epsilon,\infty)$ in \eqref{ch5eq:FLQ} to get
\begin{equation*}
  \wbr{\Q}_{i}(t)([\epsilon,\infty))=
    \wbr{Q}_{i}(t) e^{-\epsilon/\E{D_{i}}}.
\end{equation*}
Observing that
$\lim\limits_{\epsilon\rightarrow 0} \frac{1-e^{-\epsilon/\E{D_{i}}}}{\epsilon}=\frac{1}{\E{D_{i}}}$, we derive
\begin{align*}
\int_{0}^{t}\lim_{\epsilon\rightarrow 0}
\frac{\wbr{Q}_{i}(s)-\wbr{\Q}_{i}(s)(\left[\epsilon,\infty\right))}
{\epsilon}ds&= \lim\limits_{\epsilon\rightarrow 0} \frac{1-e^{-\epsilon/\E{D_{i}}}}{\epsilon}
\int_{0}^{t} \wbr{Q}_{i}(s)ds\\
&=\frac{1}{\E{D_{i}}}\int_{0}^{t} \wbr{Q}_{i}(s)ds.
\end{align*}
\end{example}

An important question is when a solution of the fluid model equations (if it exists) is unique. The next theorem answer this question.
\begin{theorem}\label{Ch5thm:Uniqueness}
Assume that\ $\wbr{Q}_{ij}(0)>0$\ if\ $\wbr{Q}_{i}(0)=K_i$ and consider the linearized Distflow power model~\ref{eq:Dist}.
Suppose that
 $\wbr{\blt{Z}}(0)=\blt{0}$ or that
$\wbr{\blt{Z}}(0)\in (0,\infty)^{I\times J}$ and the first projection of $\wbr{\blt{\Z}}(0)$ is Lipschitz continuous, i.e., there exists $L>0$ such that for any $i,j\geq 1$, $x<x'$, and $y>0$,
\begin{equation*}
\wbr{\Z}_{ij}(0)\left(
 [x,x']\times [y,\infty)
 \right) \leq
 L(x'-x).
\end{equation*}
Then there exists a unique solution of the fluid model equations.
\end{theorem}
The proof of Theorem~\ref{Ch5thm:Uniqueness} is given in Section~\ref{ch5:proofof Unique} and the main steps of the proof are as follows.
\begin{enumerate}
\item The first step is to show that each pair $\left(\wbr{Q}_{i}(\cdot),\wbr{R}_{i}(\cdot)\right)$
      satisfies a one-dimensional reflection mapping and (each pair) is unique.
\item Second, we show that $\wbr{Z}_{ij}(t)>0$ for any $i,j\geq 1$ and $t>0$.
\item Last, we prove that $(\wbr{\blt{\Z}}(\cdot),\wbr{\blt{Z}}(\cdot))$
is also unique using arguments from \cite{remerova14}.
\end{enumerate}

Having defined a fluid model and studied its properties, the next main step is to show that the fluid model arises as a weak limit of the original stochastic model under an appropriate scaling. This is the topic of the next section.

\section{Fluid limit theorem}\label{Ch5: Fluid approximation}
In this section, we study the asymptotic behavior of the stochastic network described in Section~\ref{Ch5:model description}. Consider a family of systems indexed by $n\in \mathbb{N}$, where $n$ tends to infinity, with the same basic structure as that of the system described in Section~\ref{Ch5:model description}.
To indicate the position of the system in the sequence of systems, a superscript $n$ will be appended to the system parameters and processes.

First, we introduce our asymptotic regime. We assume that the scaled capacity at node $i$ is given by $M_i^n=n M$, the scaled number of EV chargers at node $i$ is $K_i^n=n K$, and the scaled resistance and reactance on line $\epsilon_{pk}$ are given by $\res_{pk}^n= \res_{pk}/n$ and $\reac_{pk}^n=\reac_{pk}/n$.
Note that in our setting we need to scale the physical parameters of the system in contrast to the typical scalings in stochastic networks that arise in
communication networks.
We summarize below the assumptions we make in this section.

\textbf{Assumptions:}
\begin{enumerate}
  \item The scaled parameters are given by $K_i^n=n K$, $M_i^n=n M$, $\res_{pk}^n= \res_{pk}/n$, and $\reac_{pk}^n=\reac_{pk}/n$.
  \item The external arrival process satisfies $\frac{E_{ij}^n(\cdot)}{n}\overset{d}\rightarrow \wbr{E}_{ij}(\cdot)$, with $\wbr{E}_{ij}(t)=\int_{0}^{t} \lambda_{ij}(s)ds$.
  \item The limit of the external arrival process is Lipschitz continuous; i.e., there exists $\eta_{ij}>0$ such that
       $|\wbr{E}_{ij}(t)-\wbr{E}_{ij}(s)|\leq \eta_{ij}|t-s| $, for $t,s \geq 0$.
  \item The scaled initial configurations converge to
             random vectors of finite measures, $\wbr{\Q}_{ij}^n(0)\overset{d}\rightarrow \wbr{\Q}_{ij}(0)$ and
$\wbr{\Z}_{ij}^n(0)\overset{d}\rightarrow \wbr{\Z}_{ij}(0)$ as $n\rightarrow \infty$.
  \item For any $i,j\geq 1$, $\wbr{\Q}_{ij}(0)(\R_+)$ and the projections
$\wbr{\Z}_{ij}(0)(\cdot \times \R_+)$ and
$\wbr{\Z}_{ij}(0)(\R_+ \times \cdot)$ are almost surely free of atoms.
\end{enumerate}

Having introduced our scaling regime, we now move to the fluid-scaled state descriptor. The fluid-scaled measure-valued processes are given by
$\left(\blt{\wbr{\Q}}^n(\cdot),
\blt{\wbr{\Z}}^n(\cdot)\right):=
\left(\frac{\blt{\Q}^n(\cdot)}{n},
\frac{\blt{\Z}^n(\cdot)}{n}\right)$
and the fluid-scaled counting processes are given by
$\left(\blt{\wbr{Q}}^n(\cdot),
\blt{\wbr{Z}}^n(\cdot)\right):=
\left(\frac{\blt{Q}^n(\cdot)}{n},
\frac{\blt{Z}^n(\cdot)}{n}\right)$.
Moreover, our fluid scaling leads to the following relation $\blt{p}^n(\blt z)=\blt{p}(\frac{\blt z}{n})$. To see the latter, observe that under our scaling the feasible set of \eqref{ROP} can be written as follows
\begin{equation*}
\mathfrak{F}^n(\blt z):=
\left\{
\begin{split}
&\sum_{j=1}^{J} \frac{z_{ij}}{n}p_{ij} \leq  M_i,\
0 \leq p_{ij} \leq c^{max}_j,\
 W_{ii}\geq\underline{ \upsilon}_i,\
W_{pp}W_{kk}-W_{pk}^2\geq 0,\\
&W_{pk}-W_{kk}-\res_{pk}
\sum_{l\in \Node(k)}\sum_{j=1}^{J}
\frac{z_{lj}}{n} p_{lj} \\
&\qquad+
\sum_{\substack{l\in \Node(k)\\ \epsilon_{ls}\in \Edge(k)}}
\left(
(W_{ll}-2W_{ls}+W_{ss})\frac
{\res_{pk}\res_{ls}+\reac_{pk}\reac_{ls}}{\res^2_{ls}+\reac^2_{ls}}
\right)
=0.
\end{split}
\right\}.
\end{equation*}
It is clear now that $\mathfrak{F}^n(\blt z)=\mathfrak{F}(\frac{\blt z}{n})$, which leads to
\begin{align*}
\blt{p}^n(\blt z)&=
\argmax\limits_{(\blt{p},\blt W)\in \mathfrak{F}^n(\blt z)}\
\sum_{i=1}^{I} \sum_{j=1}^{J} z_{ij} u_{ij}(p_{ij})\\
&=
\argmax\limits_{(\blt{p},\blt W)\in \mathfrak{F}(\frac{\blt z}{n})}\
\sum_{i=1}^{I} \sum_{j=1}^{J} \frac{z_{ij}}{n} u_{ij}(p_{ij})=
\blt{p}(\frac{\blt z}{n}).
\end{align*}
Furthermore, by \eqref{eq:service}, we have that
\begin{equation*}
S_{ij}^n(\blt{Z}^n,s,t)=S_{ij}(\wbr{\blt{Z}}^n,s,t).
\end{equation*}

The next theorem states that the fluid model arises as a limit of the fluid-scaled state descriptor under our assumptions.
\begin{theorem}[Fluit limit]\label{Ch5thm:fluid limit}
The sequence of the fluid-scaled measure-valued vector process
$\left(\blt{\wbr{\Q}}^n(\cdot),
\blt{\wbr{\Z}}^n(\cdot)\right)$ is tight and every accumulation point $\left(\blt{\wbr{\Q}}(\cdot),
\blt{\wbr{\Z}}(\cdot)\right)$ is a fluid model solution.
\end{theorem}

When the allocation mechanism is given by the linearized Distflow power model~\ref{eq:Dist}, we can invoke Theorem \ref{Ch5thm:Uniqueness} to strengthen this result to a convergence result.
For the full AC case, the same can be concluded if the bandwidth allocation function is Lipschitz continuous. The proof of Theorem~\ref{Ch5thm:fluid limit} is given in
Section~\ref{Ch5:proof fluid lilit}, which is organized as follows.
\begin{enumerate}
  \item We establish tightness of the associated fluid-scaled measure-valued vector process $\left(\blt{\wbr{\Q}}^n(\cdot),
\blt{\wbr{\Z}}^n(\cdot)\right)$.
  \item We then show tightness for the fluid-scaled stochastic process describing the number of rejected customers, i.e., $\wbr{\blt{R}}^n(\cdot)$.
  \item The last step is to show that the limit of any convergent subsequence of
$\left(\blt{\wbr{\Q}}^n(\cdot),
\blt{\wbr{\Z}}^n(\cdot)\right)$ satisfies the fluid model equations.
\end{enumerate}
\begin{remark}
The fluid limit theorem holds even if the external arrival process is a process with a general mean $\br{E}_{ij}(\cdot)$. In this case, we need to modify the definition of a fluid model solution such that
$\wbr{R}_{ij}(t)=\int_{0}^{t}
\ind{\sum_{j=1}^{J}\wbr{Q}_{ij}(s)}
d\wbr{R}_{ij}(s)
$. However, it seems that the uniqueness of the fluid model solutions does not hold.
\end{remark}

In this section, we have obtained a fluid limit that holds for a general tree network. In the next section, we investigate under what assumptions the fluid limit converges to an invariant point.

\section{Invariant analysis}\label{Ch5: Convergence to IP}
In this section, we study the behavior of the system as time goes to infinity. To do so, we assume that the arrival rate is constant, i.e.,
$\wbr{E}_{ij}(t)=\lambda_{ij}t$.

First, we prove a result equivalent to \cite[Theorem~3.6]{kang2015}. There, a characterization of the invariant point for a loss system is shown, which we now prove in our setting. However, the difference with \cite{kang2015} is that we consider different test functions to define the fluid model and second, we consider multiple types of customers. All proofs are gathered in Section~\ref{ch5:proofs conv to IP}.

Define the traffic intensity at node $i$ of type-$j$ EVs by $\rho_{ij}:=\lambda_{ij}\E{D_{ij}}$ and the total traffic  intensity an node $i$  by $\rho_i:=\sum_{j=1}^{J}\rho_{ij}$. The following result characterizes the invariant states of a loss system with multiple types of EVs.
\begin{proposition}\label{Ch5prop:invpointQ}
Let $\lambda_{ij}\in (0, \infty)$. We have that
$(\blt{\Q}^*, \blt{q}^*)$ is invariant
if and only if for any Borel set $A\in \mathcal{B}(\R_+)$ and
$i,j\geq 1$,
\begin{equation*}
\Q_{ij}^*(A)= \frac{\lambda_{ij}}{\rho_{i}} (\rho_{i}\Min K_i)
 \int_{0}^{\infty} \Prob{D_{ij}\in A+s}ds,
\end{equation*}
and $q_{ij}^*=\frac{\rho_{ij}}{\rho_{i}} (\rho_{i}\Min K_i)$.
\end{proposition}
In the sequel, we examine the asymptotic behavior of
$(\wbr{\blt{\Z}}(\cdot), \wbr{\blt{ Z}}(\cdot))$. We make an additional assumption that the network is monotone as it is stated in the following definition.
\begin{definition}\label{ch5:defmonotone}
An allocation mechanism is called ``monotone'' if\quad $0<\blt{y}\leq \blt{z}$ implies that
$p_{ij}(\blt{y})\geq p_{ij}(\blt{z}) $.
\end{definition}
For instance, this property holds when the network has a line topology under the linearized Distflow model described in  Section~\ref{sec:Distflow}. In this case, \cite[Proposition~5]{BorstEgorovaZwart2014} can be applied directly in order to show the desired monotonicity property.  We conjecture that this monotonicity holds true for a line network under the AC power flow model as well, but have not been able to prove this, apart from the case of two nodes.
\begin{proposition}\label{Ch5prop:invpointZ}
If the network is monotone, then we have that $(\wbr{\blt{\Z}}(t), \wbr{\blt{ Z}}(t))
\rightarrow (\blt{\Z}^*,\blt{z}^*)$ as $t\rightarrow \infty$.
Furthermore, the vector $(\blt{\Z}^*,\blt{z}^*)$ satisfies the following relation. For any Borel set $A'\in \mathcal{B}(\R_+^2)$ and $i,j\geq 1$,
\begin{equation*}
\Z_{ij}^*(A')= \frac{\lambda_{ij}}{\rho_{i}} (\rho_{i}\Min K_i)
 \int_{0}^{\infty} \Prob{(B_{ij},D_{ij})\in A'+\left(p_{ij}(\blt{z}^*)s, s\right)}ds,
\end{equation*}
and $\blt{z}^*$ is given by the solution of the fixed-point equation
\begin{equation}\label{ch5eq:invariant}
z_{ij}^*=\frac{\lambda_{ij}}{\rho_i}(\rho_i\Min K_i)
\E{D_{ij}\Min \frac{B_{ij}}{p_{ij}(\blt{z}^*)}}.
\end{equation}
\end{proposition}
The proof of the last proposition combines Proposition~\ref{Ch5prop:invpointQ} and arguments from \cite[Theorem~2]{remerova14}.
Moreover, using similar arguments from \cite[Theorem~6]{remerova14} and
\cite[Theorem~3.3]{kang2012asymptotic}, it
 can be shown that the fluid and steady-state limits can be interchanged and
the sequence of fluid-scaled stationary distributions
$(\wbr{\blt{\Z}}^n(\infty), \wbr{\blt{Z}}^n(\infty))$ converges weakly to the invariant point as $n\rightarrow \infty$, provided that the invariant point is unique.

The invariant point $\blt{z}^*$, when it is unique, can be computed by solving a single ACOPF problem, which is convex in our case since it admits an exact convex relaxation.
Define the functions
\begin{equation}\label{eq:defg}
g_{ij}(x):=\frac{\lambda_{ij}}{\rho_i}(\rho_i\Min K_i) \E{ \min \{D_j x , B_{j}\}},
\end{equation}
and recall that the aggregated  allocation (the total power which type-$j$ EVs consume) at node $i$ is $\Lambda_{ij}(\blt z):=z_{ij} p_{ij}(\blt z)$. Also, for a random variable $Y$, denote by $\inf(Y)$ the leftmost point of its support.
\begin{proposition}[Characterization of the invariant point]\label{Th:UnFP}
Let $\inf{(D_j/B_{j})}\leq 1/c_j^{\text{max}}$.
The solution $\blt z^*$ of \eqref{ch5eq:invariant} is unique and is given by
$z_{ij}^*=\frac{\Lambda_{ij}^*}{g_{ij}^{-1}(\Lambda_{ij}^*)}$, where $\blt \Lambda^*$
is the unique solution of the optimization problem
 \begin{equation}\label{GDOP}
\begin{aligned}
& \underset{}{\max\limits_{\blt{\Lambda},\blt W}}
& & \sum_{i=1}^{I} \sum_{j=1}^{J}  G_{ij} (\Lambda_{ij}) \\
& \text{subject to}
&&W_{ik}-W_{kk}-P_{\Node(k)} \res_{ik}-Q_{\Node(k)} \reac_{ik}=0,\\
&&&\underline{ \upsilon}_i \leq  W_{ii}
\leq \overline{ \upsilon}_i, \blt{W}(\ep_{ik})  \succeq  0,\
\  \Lambda_{ij} \leq M_i,\\
&&&  0 \leq \Lambda_{ij} \leq g_{ij}(c_j^{max}),\ \ep_{ik}\in \Edge.
\end{aligned}
\end{equation}
Furthermore, $G_{ij}(\cdot)$ is a strictly concave function such that $G_{ij}'(\cdot)=u_{ij}'(g^{-1}_{ij}(\cdot))$ for any $i,j\geq 1$.
\end{proposition}
By \eqref{eq:power}, observe that $W_{kk}$ depends on $\blt{z}$ through the products $z_{ij} p_{ij}(\blt z)$. By the definition of $\blt \Lambda$, we have that $W_{kk}$ depends only on $\blt \Lambda$. That is, the previous optimization problem is indeed independent of the fixed point $\blt z^*$.
Note that when the assumption $\inf{(D_j/B_{j})}\leq 1/c_j^{\text{max}}$ is violated, there can be a continuum of invariant fluid model solutions \cite{remerova14}.
To arrive at \eqref{GDOP}, the essential idea is to add Little's law \eqref{ch5eq:invariant} to the set of Karush-Kuhn-Tucker (KKT) conditions that characterize
$\blt p ( \blt z)$ and rewrite all equations in such a way that they form the KKT conditions for the problem \eqref{GDOP}. The proof of Proposition~\ref{GDOP} is given in \cite[Theorem~1]{aveklourisstochastic}.

A natural question is if Proposition~\ref{Ch5prop:invpointZ} holds for non-monotone networks. However, this is an open problem even in the area of communication networks; see \cite{remerova14} and \cite{BorstEgorovaZwart2014}. In bandwidth-sharing networks, the feasible set is polyhedral. It is further proved that if the network topology is radial, then the monotonicity property holds
\cite[Proposition~5]{BorstEgorovaZwart2014}. Unfortunately, this is not the case in our setting. Distribution  networks are not in general monotone. More surprisingly the monotonicity property for the tree networks does not hold even if we consider a polyhedral feasible set, i.e., linearized Distflow model. In the next section, we explain  the reason why the monotonicity property fails.

\subsection{A counterexample of monotonicity for a general tree network}
A line network is monotone under the linearized Distflow model as we have already discussed and we conjecture this property is true for the AC power flow model as well.
However, if we extend the line network to a tree network, the monotonicity property may fail to hold for both power flow models. Below we present a counterexample for both of power flow models.

Assume a tree network with four nodes, i.e., one feeder and three load nodes. The feeder (node 0) is connected to node 1, which has two children, nodes 2 and 3. Moreover, assume that the network is constant in the sense that all the resistances and reactances are the same for all the lines and take $\res_{pk}=\reac_{pk}=0.1$ for any edge $\epsilon_{pk}$. Further, the voltage magnitude at the feeder is fixed and taken to be $W_{00}=1$ and the lower bound for the voltage magnitude is given by $\underline{\upsilon}=(0.9)^2=0.81$.
We shall show numerically that monotonicity does not hold for this simple tree network using the proportional fairness allocation mechanism. Indeed, we solve the optimization problem for the vectors
$\blt{z}=(1,1,1)$ and $\blt{y}=(1,2,1)$. The allocated power to cars is given in Table~\ref{table:conterE1},
\begin{table}[!h]
\begin{center}
\caption{Allocated power to EVs}
\label{table:conterE1}
\begin{tabular}{ |c|c|c|c| }
 \hline
  & Node~1 & Node~2& Node~3 \\
  \hline
 $\blt{p}(\blt{z})$ & $0.3050$ & $0.2008$& $0.2008$ \\
 $\blt{p}(\blt{y})$ & $0.2297$ & $0.1148$& $0.2177$ \\
 \hline
 \end{tabular}
\end{center}
\end{table}
where we observe that $p_3(\blt {y})>p_3(\blt {z})$ and hence the network is not monotone. The voltage magnitudes are given in Table~\ref{table:conterE2}.
\begin{table}[!h]
\begin{center}
\caption{Voltage magnitudes}
\label{table:conterE2}
\begin{tabular}{ |c|c|c|c| }
 \hline
  & Node~1 & Node~2& Node~3 \\
  \hline
 $\blt{W}(\blt{z})$ & $0.8507$ & $0.8100$& $0.8100$ \\
 $\blt{W}(\blt{y})$ & $0.8566$ & $0.8100$& $0.8124$ \\
 \hline
 \end{tabular}
\end{center}
\end{table}

The intuition behind this counterexample is as follows.
The voltage constraints at the leaf nodes are both active for vector $\blt{z}$ as the network is constant. The voltage magnitude constraint remains active in the node where we increase the number of EVs, i.e., node 2 in this example. The total allocated power at node~2 is $\Lambda_2(\blt{y})=2*0.1148=0.2296>0.2008=\Lambda_2(\blt{z})$. However, the feeder can not allocate more power (than $\Lambda_2(\blt{y})$) to node~2 because of the voltage drop constraints. As a result, the feeder allocates more power to node~3 and hence $p_3(\blt {y})>p_3(\blt {z})$, even though the number of EVs at node~3 does not increase. To see this, we remove the voltage drop constraints from the model and solve again the optimization problem; see Tables~\ref{table:conterE3} and \ref{table:conterE4}. Observe now that $\blt{p}(\blt{z})\geq \blt{p}(\blt{y})$ and the voltage magnitude at node~2 decreases.
\begin{table}[!h]
\begin{center}
\caption{Allocated power to EVs without voltage constraints}
\label{table:conterE3}
\begin{tabular}{ |c|c|c|c| }
 \hline
  & Node~1 & Node~2& Node~3 \\
  \hline
 $\blt{p}(\blt{z})$ & $0.9799$ & $0.5708$& $0.5708$ \\
 $\blt{p}(\blt{y})$ & $0.7668$ & $0.3643$& $0.5156$ \\
 \hline
 \end{tabular}
\end{center}
\end{table}

\begin{table}[!h]
\begin{center}
\caption{Voltage magnitudes without voltage constraints}
\label{table:conterE4}
\begin{tabular}{ |c|c|c|c| }
 \hline
  & Node~1 & Node~2& Node~3 \\
  \hline
 $\blt{W}(\blt{z})$ & $0.3457$ & $0.2164$& $0.2164$ \\
 $\blt{W}(\blt{y})$ & $0.3763$ & $0.2047$& $ 0.2631$ \\
 \hline
 \end{tabular}
\end{center}
\end{table}
Even more surprisingly, the same behavior holds even if we use the linearized Distflow model. To see that, observe that by \eqref{eq:Dist}, the voltages at leaf nodes have an explicit solution, namely
\begin{align*}
W_{22}^{lin}(\blt{\Lambda})=W_{00}-0.2 (\Lambda_1+2\Lambda_2+\Lambda_3),\\
W_{33}^{lin}(\blt{\Lambda})=W_{00}-0.2 (\Lambda_1+\Lambda_2+2\Lambda_3).
\end{align*}
In other words, the voltages at leaf nodes depend on the power which is allocated to all three nodes. This is not the case for the constraints in bandwidth-sharing networks studied in \cite{BorstEgorovaZwart2014}. Constraints like the above correspond to non-tree networks in their setting. Hence, the desired monotonicity property does not hold in general in our model. Our intuition agrees with the numerical results in Table~\ref{table:conterE5}.
\begin{table}[!h]
\begin{center}
\caption{Allocated power to EVs for the linearized Distflow model}
\label{table:conterE5}
\begin{tabular}{ |c|c|c|c| }
 \hline
  & Node~1 & Node~2& Node~3 \\
  \hline
 $\blt{p}(\blt{z})$ & $0.3167$ & $0.2111$& $0.2111$ \\
 $\blt{p}(\blt{y})$ & $0.2375$ & $0.1188$& $0.2375$ \\
 \hline
 \end{tabular}
\end{center}
\end{table}

\section{Proofs for Section~\ref{ch5:Perturbation}}\label{ch5:proofs of pertupation}

\begin{proof}[Proof of Proposition~\ref{prop:feasible}]
First, note that the point $\blt{0}$ lies in the feasible set by choosing $W_{pk}=W_{00}$. We now define a partition of the set $\Node$. Recall that $\Node(k)$ denotes the subtree rooted in node $k \in \Node$ (including node $k$).
Let us define the following sets
$\mathcal{L}_0:=\left\{k\in \Node: \Node(k)= \{k\} \right\}$ and for any $m\geq 1$,
\begin{equation*}
\mathcal{L}_m:=\left\{k\in \Node\setminus \bigcup_{n=0}^{m-1}\mathcal{L}_n:
\Node(l)\subseteq
\bigcup_{n=0}^{m-1}\mathcal{L}_n \cup \{k\} \right\}.
\end{equation*}
As the number of nodes
$I+1$ is finite, there exists $I'\leq I+1$ such that $\mathcal{L}_{I'}=\{0\}$, i.e., $\mathcal{L}_{I'}$ contains only the feeder node.
Note that $\mathcal{L}_0$ is the set of leaf nodes and the family
$\mathcal{L}:=\{\mathcal{L}_m,  0\leq m\leq I'\}$
is a  partition of the set $\Node$. Indeed, we have that $\emptyset \notin \mathcal{L}$, $\bigcup_{m=0}^{I'}\mathcal{L}_m=\Node$, and $\mathcal{L}_i \cap \mathcal{L}_k=\emptyset $ for $i\neq k$.
In Figure~\ref{fig:parition}, we depict an example of a partition with five sets.
\begin{figure}[!h]
\centering
\includegraphics[scale=0.5]{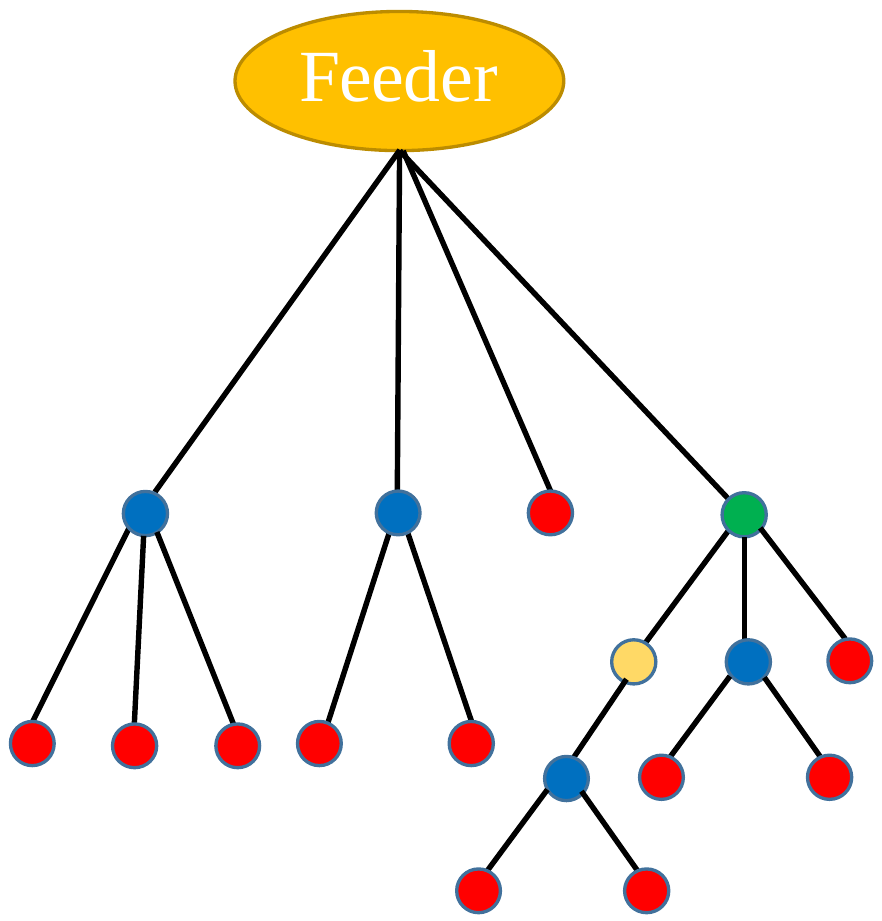}
\caption{The sets $\mathcal{L}_i$ in a tree network. In this case $I'=4$. The red nodes are in  $\mathcal{L}_0$, the blue nodes are in $\mathcal{L}_1$, the yellow node is in  $\mathcal{L}_2$, the green node is in $\mathcal{L}_3$, and $\mathcal{L}_4$ includes only the feeder.}
\label{fig:parition}
\end{figure}

Without loss of generality, we consider a single type of EVs; otherwise set $\Lambda_k:=\sum_{j=1}^{J}\Lambda_{kj}$. To simplify the notation, in the rest of the proof we write $\blt{\Lambda}$ instead of $\blt{\Lambda}(\blt{z})$ and $W_{kk}$ instead of $W_{kk}(\blt{\Lambda})$.
Recalling that $\blt{\Lambda}$ is a feasible point of \eqref{ROPband}, we have that $\Lambda_k\leq M_k$, $\Lambda_k\leq z_k c^{\max}$ and for $k\geq 1$, $\epsilon_{pk}\in \Edge$,
\begin{equation}
\begin{split}
W_{pk}-W_{kk}-P_{\Node(k)} \res_{pk}-Q_{\Node(k)} \reac_{pk}=0, \label{eq:FeasNet1} \\
\underline{ \upsilon}_k \leq W_{kk}
\leq \wbr{\upsilon}_k,\\
W_{pp} W_{kk}
- W_{pk}^2\geq 0.
\end{split}
\end{equation}

Clearly, $\blt{\Lambda}'$ satisfies the linear constraints of \eqref{ROPband}, i.e.,
$\Lambda_k'\leq \Lambda_k\leq M_k$  and
$\Lambda_k'\leq\Lambda_k\leq z_k c^{\max}$. To show that $\blt{\Lambda}$ is a feasible point of \eqref{ROPband}, we need to construct $W_{il}'$, $i,l\geq 0$ such that the additional constraints of \eqref{ROPband} are satisfied if we replace $\blt{\Lambda}$ by  $\blt{\Lambda}'$. To this end, set $W_{00}'=W_{00}$,
$W_{pk}'=W_{pk}$,
 for $\epsilon_{pk}\in \Edge$.
Further,  $W_{kk}'$ for $k\geq 1$, are given by the solution of
\begin{equation}\label{eq:FeasNet2}
W_{pk}'-W_{kk}'-P_{\Node(k)}' \res_{pk}-Q_{\Node(k)}' \reac_{pk}=0,\ \epsilon_{pk}\in \Edge.
\end{equation}

We shall show that $W_{kk}\leq W_{kk}'$ for $k \in  \Node$. The proof is then concluded by observing that by the inequality
$W_{kk}
\leq W_{kk}'$,
we have that
$\underline{ \upsilon}_k \leq W_{kk}(\blt{\Lambda}')$
for $k\geq 1$. Furthermore, by the third equation of \eqref{eq:FeasNet1}, we get  for $\epsilon_{pk}\in \Edge$,
\begin{equation*}
\begin{split}
W_{pp}'W_{kk}'-W_{pk}'^2
=  W_{pp}'W_{kk}'
-W_{pk}^2
&=W_{pp}'W_{kk}'
-W_{pp}W_{kk}\\
&\geq W_{pp}(W_{kk}'-W_{kk})\geq 0.
\end{split}
\end{equation*}
Thus, $\blt{\Lambda}'$ satisfies all the constraints  of \eqref{ROPband}, and hence it is a  feasible point.

We now proceed to the proof of the claim that $W_{kk}\leq W_{kk}'$ for $k \in  \Node$.
Define $a_{pkls}:=\frac{\res_{pk}\res_{ls}+\reac_{pk}\reac_{ls}}{\res_{ls}^2+\reac_{ls}^2}$ and
$\Node(k)^-:=\Node(k)\setminus \{k\}$.
For $k\in \mathcal{L}_{m}$ for some $0\leq m < I'$ we have that
\begin{align*}\label{eq:FeasNet3}
W_{kk}-W_{kk}'=& \res_{pk} (P_{\Node(k)}'-P_{\Node(k)}) + \reac_{pk} (Q_{\Node(k)}'-Q_{\Node(k)})\\
=&\res_{pk}\sum_{l\in \Node(k)}(\Lambda_l'-\Lambda_l)
+\sum_{l\in \Node(k)}\sum_{\epsilon_{ls}\in \Edge(k)}
a_{pkls} (W_{ll}'-W_{ll}+W_{ss}'-W_{ss}).
\end{align*}
The last equation can be rewritten as follows
\begin{equation}\label{eq:FeasNet4}
\begin{split}
(1+\sum_{\epsilon_{ks}\in \Edge(k)} a_{pkks})
(W_{kk}-W_{kk}')=
\res_{pk}\sum_{l\in \Node(k)^-}(\Lambda_l'-\Lambda_l)
+\res_{pk}(\Lambda_k'-\Lambda_k)\\
+\sum_{l\in \Node(k)^-}\sum_{\epsilon_{ls}\in \Edge(k)}
a_{pkls} (W_{ll}'-W_{ll}+W_{ss}'-W_{ss})
+\sum_{\epsilon_{ks}\in \Edge(k)}a_{pkks}(W_{ss}'-W_{ss}).
\end{split}
\end{equation}
We now show the inequality $W_{kk}\leq W_{kk}'$ for each $k$ by induction.
Let
$k\in \mathcal{L}_{0}$. By \eqref{eq:FeasNet4}, we have that
\begin{equation}\label{eq:L0}
W_{kk}-W_{kk}'= \res_{pk}(\Lambda_k'-\Lambda_k)\leq 0,
\end{equation}
where $p$ is the unique parent of node $k$.
If $m=1$ (i.e., $k \in \mathcal{L}_1$), then we have that
$\Node(k)^-=\Node(k)\setminus \{k\}=\mathcal{L}_0\cap\Node(k)\setminus \{k\}$ and
$\{\epsilon_{ls}\in \Edge(k):  l\in \Node(k)\setminus \{k\}\}=\emptyset$. Further,
$\{s: \epsilon_{ks}\in \Edge(k)\}=\mathcal{L}_0\cap\Node(k)\setminus \{k\}$.
By \eqref{eq:FeasNet4} and \eqref{eq:L0}, we obtain
\begin{equation}\label{eq:L1}
\begin{split}
(1+\sum_{s\in \mathcal{L}_0\cap\Node(k)^-} a_{pkks})
(W_{kk}-W_{kk}')=\res_{pk}(\Lambda_k'-\Lambda_k)\\
+\sum_{l\in \mathcal{L}_0\cap\Node(k)^-}(\res_{pk}-a_{pkkl}\res_{kl})(\Lambda_l'-\Lambda_l).
\end{split}
\end{equation}
Now, observe that
\begin{equation*}
\begin{split}
\res_{pk}-a_{pkkl}\res_{kl}&=
\res_{pk}-\res_{kl} \frac{\res_{pk}\res_{kl}+\reac_{pk}\reac_{kl}}{\res_{kl}^2+\reac_{kl}^2}\\
&= \left(\res_{kl}^2+\reac_{kl}^2\right)^{-1}
\left(\res_{pk}\res_{kl}^2+ \res_{pk}\reac_{kl}^2 -\res_{pk}\res_{kl}^2-\res_{kl}\reac_{pk}\reac_{kl}  \right)\\
&= \reac_{kl}\left(\res_{kl}^2+\reac_{kl}^2\right)^{-1}
\left(\res_{pk}\reac_{kl}- \res_{kl}\reac_{pk} \right)= 0,
\end{split}
\end{equation*}
where the last equation holds by the assumption that $\frac{\res_{pk}}{\reac_{pk}}$ is constant.
That is, $
W_{kk}\leq W_{kk}'$, for $k\in \mathcal{L}_1$. Suppose now that $k \in \mathcal{L}_2$. By \eqref{eq:FeasNet4}, we have that
\begin{equation*}
\begin{split}
(1+\sum_{\epsilon_{ks}\in \Edge(k)} a_{pkks})
(W_{kk}-W_{kk}')&=
\res_{pk}\sum_{m=0}^{1}\
\sum_{l\in\mathcal{L}_m \cap \Node(k)^-}(\Lambda_l'-\Lambda_l)
+\res_{pk}(\Lambda_k'-\Lambda_k)\\
&+\sum_{l\in \mathcal{L}_1 \cap\Node(k)^-}
\sum_{\substack{ \epsilon_{ls}\in \Edge(k)\\  s\in \mathcal{L}_0\cap \Node(l)}}
a_{pkls} (W_{ll}'-W_{ll}+W_{ss}'-W_{ss})\\
&+\sum_{m=0}^{1}
\sum_{\substack{
\epsilon_{ks}\in \Edge(k)\\ s \in \mathcal{L}_m}}
a_{pkks}(W_{ss}'-W_{ss}).
\end{split}
\end{equation*}
The last equation can be equivalently rewritten as follows
\begin{equation*}
\begin{split}
(1+\sum_{\epsilon_{ks}\in \Edge(k)} a_{pkks})
(W_{kk}-W_{kk}')&=
\res_{pk}\sum_{m=0}^{1}\
\sum_{l\in\mathcal{L}_m \cap \Node(k)^-}(\Lambda_l'-\Lambda_l)
+\res_{pk}(\Lambda_k'-\Lambda_k)\\
&+\sum_{l\in \mathcal{L}_1 \cap\Node(k)^-}
(\sum_{\substack{ \epsilon_{ls}\in \Edge(k)\\  s\in \mathcal{L}_0\cap \Node(l)}}
a_{pkls}+a_{pkkl}) (W_{ll}'-W_{ll})\\
&+\sum_{l\in \mathcal{L}_1 \cap\Node(k)^-}
\sum_{\substack{ \epsilon_{ls}\in \Edge(k)\\  s\in \mathcal{L}_0\cap \Node(l)}}
a_{pkls}(W_{ss}'-W_{ss})\\
&+
\sum_{\substack{
\epsilon_{ks}\in \Edge(k)\\ s \in \mathcal{L}_0}}
a_{pkks}(W_{ss}'-W_{ss}).
\end{split}
\end{equation*}
Applying \eqref{eq:L1} in the last equation, we obtain the following relation
\begin{equation*}
\begin{split}
&(1+\sum_{\epsilon_{ks}\in \Edge(k)} a_{pkks})
(W_{kk}-W_{kk}')=
\res_{pk}(\Lambda_k'-\Lambda_k)
+\sum_{\substack{
\epsilon_{ks}\in \Edge(k)\\ s \in \mathcal{L}_0}}
(\res_{pk}-\res_{ks}a_{pkks})(\Lambda_{s}'-\Lambda_{s})
\\
&+\sum_{l\in\mathcal{L}_1 \cap \Node(k)^-}
\left(
\res_{pk}-
\res_{kl}(1+\sum_{s\in \mathcal{L}_0\cap\Node(l)^-} a_{klls})^{-1}
(\sum_{\substack{ \epsilon_{ls}\in \Edge(k)\\  s\in \mathcal{L}_0\cap \Node(l)}}
a_{pkls}+a_{pkkl})
\right)
(\Lambda_l'-\Lambda_l)\\
&+
\sum_{\substack{l\in \mathcal{L}_1 \cap\Node(k)^-\\
\epsilon_{ls}\in \Edge(l)\\ s \in \mathcal{L}_0}} \left(
\res_{pk}
-\res_{ls}a_{pkls}\right)
(\Lambda_{s}'-\Lambda_{s}).
\end{split}
\end{equation*}
Now, observe that using the assumption that $\frac{\res_{pk}}{\reac_{pk}}$ is the same for all edges, we have that  $\res_{pk}-\res_{ks}a_{pkks}=0$. Further, we have that
\begin{equation*}
\begin{split}
&\res_{pk}-
\res_{kl}(1+\sum_{s\in \mathcal{L}_0\cap\Node(l)^-} a_{klls})^{-1}
(\sum_{\substack{ \epsilon_{ls}\in \Edge(k)\\  s\in \mathcal{L}_0\cap \Node(l)}}
a_{pkls}+a_{pkkl})\\
&=(1+\sum_{s\in \mathcal{L}_0\cap\Node(l)^-} a_{klls})^{-1}
\left(
\res_{pk}
(1+\sum_{s\in \mathcal{L}_0\cap\Node(l)^-} a_{klls})-
\res_{kl}(\sum_{\substack{ \epsilon_{ls}\in \Edge(k)\\  s\in \mathcal{L}_0\cap \Node(l)}}
a_{pkls}+a_{pkkl})
\right)\\
&=(1+\sum_{s\in \mathcal{L}_0\cap\Node(l)^-} a_{klls})^{-1}
\left(
\res_{pk}-\res_{kl}a_{pkkl}+
\sum_{s\in \mathcal{L}_0\cap\Node(l)^-}
(\res_{pk} a_{klls}-\res_{kl}a_{pkls})
\right)=0.
\end{split}
\end{equation*}
Thus, recalling that
$\Lambda_k'-\Lambda_k\leq 0$ for any $k\in \Node$, we derive that
$W_{kk}-W_{kk}'\leq 0$ for $k \in \mathcal{L}_2$.
Suppose now that for all $k\in \mathcal{L}_j$, $j=0,\ldots,m$,
\begin{equation}\label{eq:induction}
(1+\sum_{\epsilon_{ks}\in \Edge(k)} a_{pkks})
(W_{kk}-W_{kk}')=\res_{pk}(\Lambda_k'-\Lambda_k).
\end{equation}
We shall show that the same holds for $k\in \mathcal{L}_{m+1}$. To this end, by \eqref{eq:FeasNet4}  and \eqref{eq:induction}, we have that
\begin{equation*}
\begin{split}
&(1+\sum_{\epsilon_{ks}\in \Edge(k)} a_{pkks})
(W_{kk}-W_{kk}')=
\res_{pk}(\Lambda_k'-\Lambda_k)\\
&+\sum_{j=0}^{m}
\sum_{
\substack{s\in\mathcal{L}_j \cap \Node(k)^-\\
l\in \bigcup_{b=j+1}^{m+1}\mathcal{L}_b}}
\Bigg(
\res_{pk}-
\res_{ls}(1+\sum_{\epsilon_{sf}\in \Edge(s)} a_{lssf})^{-1}
\Big(\sum_{\substack{ \epsilon_{sf}\in \Edge(s)\\
f\in \bigcup_{b=0}^{j-1}\mathcal{L}_b}}
a_{pksf}+a_{pkls}
\Big)\Bigg)
(\Lambda_s'-\Lambda_s).
\end{split}
\end{equation*}
Using again the assumption that $\frac{\res_{pk}}{\reac_{pk}}$ is the same for all lines, we obtain
$$
(1+\sum_{\epsilon_{ks}\in \Edge(k)} a_{pkks})
(W_{kk}-W_{kk}')=\res_{pk}(\Lambda_k'-\Lambda_k),
$$
for $k\in \mathcal{L}_{m+1}$.
Thus, $W_{kk}\leq W_{kk}'$ for any $k \in \mathcal{L}_{m}$, $0\leq  m\leq I'$ or
$k \in  \bigcup_{m=0}^{I'}\mathcal{L}_m=\Node$.
This concludes the proof.
\end{proof}

\begin{proof}[Proof of Theorem~\ref{Ch5Pr:continuous}]
We follow the argument in \cite[Lemma~7.1] {ReedZwart2014}.
 Take a sequence $\blt {z}^k\in (0,\infty)^{I\times J}$ such that
$\blt{z}^k\rightarrow \blt{z} $ as $k\rightarrow \infty$.
We proceed by contradiction.
Let us assume that $\blt{\Lambda}(\cdot)$ is not continuous at point $\blt{z}$. That is
$\blt{\Lambda}(\blt{z}^k) \rightarrow  \blt{\Lambda}'$ and
$\blt{\Lambda}'\neq \blt{\Lambda}(\blt{z})$.
 Note the limit $\blt{\Lambda}'$ exists as the sequence $\blt{\Lambda}(\blt{z}^k)$ lives in a subset of the compact set $\{\blt{\Lambda}\in [0,\infty)^{I\times J}: \blt{\Lambda}\leq \blt{M} \}$.
First, we show that $\blt{\Lambda}'$  is a feasible point of \eqref{ROPband}. As $\blt{\Lambda}(\blt{z}^k)$ is the optimal solution of \eqref{ROPband}, replacing $\blt{z}$ by
$\blt{z^k}$ we have that
$\sum_{j=1}^{J} \Lambda_{ij}(\blt{z}^k) \leq M_i$ and
$ 0 \leq \Lambda_{ij}(\blt{z}^k) \leq c^{max}_j z_{ij}^k $. Taking the limit as
$k\rightarrow \infty$, we derive
$\sum_{j=1}^{J} \Lambda_{ij}' \leq M_i$ and
$ 0 \leq \Lambda_{ij}' \leq c^{max}_j z_{ij} $. Further, we have that
$ W_{ii}(\blt{\Lambda}(\blt{z}^k)) \geq \underline{\upsilon}_i  $ and
$\blt{W}(\epsilon_{il},\blt{\Lambda}(\blt{z}^k)) \succeq  0,\ \epsilon_{il}\in \Edge$. The latter is equivalent to
$W_{ii}(\blt{\Lambda}(\blt{z}^k)) W_{ll}(\blt{\Lambda}(\blt{z}^k))-W_{il}(\blt{\Lambda}(\blt{z}^k))\geq 0$ (as we assume $\underline{ \upsilon}_i>0$). Now, by continuity  of the voltage magnitudes \cite[Theorem~3]{dvijotham2017high} we obtain
$  W_{ii}(\blt{\Lambda}')\geq \underline{ \upsilon}_i$
and
$W_{ii}(\blt{\Lambda}') W_{ll}(\blt{\Lambda}')-W_{il}(\blt{\Lambda}')^2\geq 0$. That is, $\blt{\Lambda}'$  is a feasible point of \eqref{ROPband}. Recalling that
$\blt{\Lambda}(\blt{z})$ is the optimal solution of \eqref{ROPband}, we have that
\begin{equation}\label{eq:contra}
\sum_{i=1}^{I} \sum_{j=1}^{J} z_{ij} u_{ij}\left(\blt {\Lambda}(\blt{z})/z_{ij}\right)
>
\sum_{i=1}^{I} \sum_{j=1}^{J} z_{ij} u_{ij}\left(\blt {\Lambda}'/z_{ij}\right).
\end{equation}

To derive the contradiction we construct a point $\blt{\Lambda}^k$ which is feasible for \eqref{ROPband} if we replace $\blt{z}$ by $\blt{z}^k$. To this end, define
for any $k\geq 1$,
\begin{equation*}
\Lambda_{ij}^{k}:=
    \Lambda_{ij}(\blt{z})\Min c^{max}_j z_{ij}^k.
\end{equation*}
We have that $\blt {\Lambda}^k\rightarrow \blt {\Lambda}(\blt{z})$ and
$\blt{\Lambda}^{k}\leq \blt{\Lambda}(\blt{z})$ for $k\geq k_0$. Observing that
$\Lambda_{ij}^k \leq c^{max}_j z_{ij}^k$, by Proposition~\ref{prop:feasible}, we have that $\blt {\Lambda}^k$ is a feasible point of \eqref{ROPband} by replacing $\blt{z}$ by $\blt{z}^k$ for $k\geq 1$. It follows that as $k\rightarrow \infty$,
\begin{equation*}
\sum_{i=1}^{I} \sum_{j=1}^{J} z_{ij}^k u_{ij}\left(\blt {\Lambda^k}/z_{ij}^k\right)
\rightarrow
 \sum_{i=1}^{I} \sum_{j=1}^{J} z_{ij} u_{ij}\left(\blt {\Lambda}(\blt{z})/z_{ij}\right)
\end{equation*}
and
\begin{equation*}
\sum_{i=1}^{I} \sum_{j=1}^{J} z_{ij} u_{ij}\left(\blt {\Lambda}(\blt{z}^k)/z_{ij}^k\right)
\rightarrow
\sum_{i=1}^{I} \sum_{j=1}^{J} z_{ij} u_{ij}\left(\blt {\Lambda}'/z_{ij}\right).
\end{equation*}
That is, by \eqref{eq:contra} there exists a sufficiently large $k$ such that
\begin{equation*}
\sum_{i=1}^{I} \sum_{j=1}^{J} z_{ij}^k u_{ij}\left(\blt {\Lambda^k}/z_{ij}^k\right)
>
\sum_{i=1}^{I} \sum_{j=1}^{J} z_{ij} u_{ij}\left(\blt {\Lambda}(\blt{z}^k)/z_{ij}^k\right).
\end{equation*}
The last inequality yields a contradiction as $\blt {\Lambda}(\blt{z}^k)$ is the optimal solution of
\eqref{ROPband} by replacing $\blt{z}$ by $\blt{z}^k$.
\end{proof}

\section{Proofs for Section~\ref{ch5: fluid model}}
\label{ch5:proofof Unique}
\begin{proof}[Proof of Proposition~\ref{ch5:analysisQ}.]
Using the identity $\Prob{D_{ij}<t}+\Prob{D_{ij}\geq t }=1$, \eqref{eq:qnet} can be written as
\begin{equation*}
\wbr{Q}_{ij}(t)=
 \wbr{Q}_{ij}(0) +\wbr{E}_{ij}(t)-\wbr{R}_{ij}(t) -\wbr{D}_{ij}(t),
\end{equation*}
where
\begin{equation}\label{eq:FLq2}
\begin{split}
\wbr{D}_{ij}(t):=
 \wbr{Q}_{ij}(0) \Prob{ D^0_{ij} < t }+
\int_{0}^{t}  \Prob{ D_{ij} < t-s } d\wbr{E}_{ij}(s)\\
-
\int_{0}^{t}  \Prob{ D_{ij} < t-s } d\wbr{R}_{ij}(s).
\end{split}
\end{equation}
In the sequel, we show that $\wbr{D}_{ij}(t)$ can be written as in \eqref{eq:depE}. By the definition of the fluid model, we have that
\begin{equation*}
\begin{split}
\wbr{Q}_{ij}(t)-\wbr{\Q}_{ij}(t)([\epsilon,\infty))&=
\wbr{Q}_{ij}(0)\left(\Prob{ D_{ij}^0 \geq t }-\Prob{ D_{ij} \in t+[\epsilon,\infty] }\right)\\
&\qquad+
\int_{0}^{t} \left(
 \Prob{ D_{ij} \geq t-s }-\Prob{ D_{ij}
\in t-s+[\epsilon,\infty] }\right)
 d\wbr{E}_{ij}(s)\\
&\qquad-
 \int_{0}^{t} \left(
 \Prob{ D_{ij} \geq t-s }-\Prob{D_{ij} \in t-s+[\epsilon,\infty] }\right)
 d\wbr{R}_{ij}(s).
\end{split}
\end{equation*}
Observing that $\Prob{ D_{ij} \in t+[\epsilon,\infty] }=
\Prob{ D_{ij} \geq t+\epsilon }$ and
\begin{equation*}
\Prob{ D_{ij} \geq t }-\Prob{ D_{ij} \geq t+\epsilon }=
\Prob{ t<D_{ij} < t +\epsilon},
\end{equation*}
we have that
\begin{equation*}
\begin{split}
\wbr{Q}_{ij}(t)-\wbr{\Q}_{ij}(t)([\epsilon,\infty))=
\wbr{Q}_{ij}(0)&\Prob{ t<D_{ij}^0 < t +\epsilon}\\
&+
\int_{0}^{t}
\Prob{ t-s<D_{ij} < t-s +\epsilon}
 d\wbr{E}_{ij}(s)\\
& -
 \int_{0}^{t}
 \Prob{ t-s<D_{ij} < t-s +\epsilon}
  d\wbr{R}_{ij}(s).
\end{split}
\end{equation*}
By the assumption of existence of the pdf $f_{D_{ij}}(\cdot)$, we have that
\begin{equation*}
\begin{split}
\wbr{Q}_{ij}(t)-\wbr{\Q}_{ij}(t)([\epsilon,\infty))=
\wbr{Q}_{ij}(0)\epsilon f_{D_{ij}^0}(t)+
\int_{0}^{t}
\epsilon f_{D_{ij}}(t-s)
 d\wbr{E}_{ij}(s)\\
 -
 \int_{0}^{t}
 \epsilon f_{D_{ij}}(t-s)
  d\wbr{R}_{ij}(s)+ o(\epsilon).
\end{split}
\end{equation*}
Dividing the last equation by $\epsilon$ and letting $\epsilon$ go to zero, 
we have that
\begin{equation}\label{eq:departurerate}
\begin{split}
\lim_{\epsilon \rightarrow 0}
\frac{\wbr{Q}_{ij}(t)-\wbr{\Q}_{ij}(t)([\epsilon,\infty))}{\epsilon}=
\wbr{Q}_{ij}(0) f_{D_{ij}^0}(t)+
\int_{0}^{t}
 f_{D_{ij}}(t-s)
 d\wbr{E}_{ij}(s)\\
 -
 \int_{0}^{t}
 f_{D_{ij}}(t-s)
  d\wbr{R}_{ij}(s).
\end{split}
\end{equation}
In other words, the limit of the left-hand side of \eqref{eq:departurerate} exists.  Integrating \eqref{eq:departurerate} from $0$ to $t$ and interchanging the integrals by using Tonelli's theorem \cite{rudin1987real}, we derive
\begin{align*}
\int_{0}^{t}
\lim_{\epsilon \rightarrow 0}
\frac{\wbr{Q}_{ij}(s)-\wbr{\Q}_{ij}(s)([\epsilon,\infty))}{\epsilon}ds=
 \wbr{Q}_{ij}(0) \Prob{ D^0_{ij} < t }+
\int_{0}^{t}  \Prob{ D_{ij} < t-s } d\wbr{E}_{ij}(s)\\
-\int_{0}^{t}  \Prob{ D_{ij} < t-s } d\wbr{R}_{ij}(s)=
\wbr{D}_{ij}(t).
\end{align*}
Furthermore, the following inequality holds for any $t\geq 0$,
\begin{align*}
\wbr{D}_{ij}(t)\leq
\int_{0}^{t}  \Prob{ D_{ij} < t-s } d\wbr{E}_{ij}(s)\leq \wbr{E}_{ij}(t)<\infty.
\end{align*}
 That is, $\wbr{D}_{ij}(t)$ represents the departure process which proves \eqref{eq:Departures} and  \eqref{eq:depE}.
\end{proof}
The first step to prove Theorem~\ref{Ch5thm:Uniqueness} is to show that the fluid model solutions are bounded away from zero. This is stated in the following proposition.
\begin{proposition}\label{ch5prop:positiveZ}
Under the assumptions of Theorem~\ref{Ch5thm:Uniqueness}, we have that
for any $\epsilon>0$,
$$\inf\limits_{t\geq \epsilon}\min\limits_{i,j}\wbr{Z}_{ij}(t)>0.$$
\end{proposition}
\begin{proof}
Recall that an assumption of Theorem~\ref{Ch5thm:Uniqueness} is that $Q_{ij}(0)>0$ if $Q_{i}(0)=K_i$. Further, by our assumptions there exists the probability density function of parking times and $f_{D_{ij}}(0)>0$ for any $i,j\geq 1$. It is enough to show that $\wbr{Z}_{ij}(\cdot)$ remains positive when the system  is not full.
Assume that $\wbr{\blt{Z}}(0)=0$ and
define
$\tau=:\{s>0: \wbr{Q}_{i}(0)=K_i \}$, where $\tau\in[0,\infty]$.
Note that
$\Prob{\frac{B_{ij}}{c_j^{max}}
\Min D_{ij}\geq s}\rightarrow
\Prob{\frac{B_{ij}}{c_j^{max}}
\Min D_{ij}\geq 0}=
 1$ as $s \rightarrow 0$ and choose $\epsilon_1$ such that
$\Prob{\frac{B_{ij}}{c_j^{max}}
\Min D_{ij}\geq s}\geq \frac{1}{2}$ for $s\in[0,\epsilon_1]$.
For $t\leq \tau$, we have that
\begin{equation*}
\begin{split}
\wbr{Z}_{ij}(t)&\geq
\int_{0}^{t} \lambda(s) \Prob{\frac{B_{ij}}{c_j^{max}}
\Min D_{ij}\geq t-s}ds\\
&=\int_{0}^{t} \lambda_{ij}(t-s) \Prob{\frac{B_{ij}}{c_j^{max}}
\Min D_{ij}\geq s}ds\\
&=\inf_{0<s\leq \epsilon}\lambda_{ij}(s)
\int_{0}^{\epsilon}
\Prob{\frac{B_{ij}}{c_j^{max}}
\Min D_{ij}\geq s}ds\\
&\geq
\inf_{0<s\leq \epsilon}\lambda_{ij}(s)
\frac{\epsilon\Min \epsilon_1}{2}>0,
\end{split}
\end{equation*}
and this covers also the case that $\tau=\infty$.
Note that if the arrival rate is constant then the last bound coincides with the one in \cite[Lemma~3]{remerova14}.
Now, for $t>\tau$, we have that $\wbr{Q}_{i}(t)=K_i$ and by the continuity of the fluid model solutions, we have that   $\wbr{Q}_{ij}(t)=\wbr{Q}_{ij}(\tau)$. Further, by \eqref{eq:Departures}, we have that
\begin{equation*}
\wbr{E}_{ij}(t)-\wbr{R}_{ij}(t)=\wbr{D}_{ij}(t)-\wbr{D}_{ij}(\tau)
+\wbr{E}_{ij}(\tau),
\end{equation*}
and using \eqref{eq:depE}, we obtain
\begin{equation*}
\begin{split}
\wbr{Z}_{ij}(t)&\geq
\int_{\tau}^{t} \delta_{ij}(s) \Prob{\frac{B_{ij}}{c_j^{max}}
\Min D_{ij}\geq t-s}ds,
\end{split}
\end{equation*}
where we define
$\delta_{ij}(s):=\lim\limits_{\epsilon\rightarrow 0}
\frac{\wbr{Q}_{ij}(s)-\wbr{\Q}_{ij}(s)\left([\epsilon,\infty)\right)}
{\epsilon}$. By the fact that
$\wbr{Q}_{ij}(t)=\wbr{Q}_{ij}(\tau)>0$ for $t>\tau$ (this also covers the case $\tau=0$), we have that $\delta_{ij}(t)=\delta_{ij}(\tau)=\delta_{ij}$. Further by the assumption $f_{D_{ij}}(0)>0$, \eqref{eq:departurerate}, and the fact that $\wbr{R}_{ij}(s)=0$ for $s\leq \tau$ we have that $\delta_{ij}(\tau)>0$.
Hence,
\begin{equation*}
\begin{split}
\wbr{Z}_{ij}(t)&\geq
\delta_{ij}\int_{\tau}^{t} \Prob{\frac{B_{ij}}{c_j^{max}}
\Min D_{ij}\geq t-s}ds\\
&=\delta_{ij}\int_{0}^{t-\tau} \Prob{\frac{B_{ij}}{c_j^{max}}
\Min D_{ij}\geq s}ds\\
&\geq \delta_{ij} \frac{(t-\tau)\Min \epsilon_1}{2}>0.
\end{split}
\end{equation*}
\end{proof}

\begin{proof}[Proof of Theorem~\ref{Ch5thm:Uniqueness}.]
We first show that each pair $(K_{i}-\wbr{Q}_{i}(\cdot),\wbr{R}_{i}(\cdot))$ is unique for any $i\geq 1$.
Note that by Remark~\ref{ch5:remark1}, fluid model solutions are invariant with respect to time shifts, and hence it suffices to show that $(K_{i}-\wbr{Q}_{i}(\cdot),\wbr{R}_{i}(\cdot))$ is unique on the time interval $[0,T]$ for $T>0$.

By Proposition~\ref{ch5:analysisQ}, we have that
\begin{equation}\label{ch5:key}
K_{i}-\wbr{Q}_{i}(t)=
K_{i}-\wbr{Q}_{i}(0)
-\sum_{j=1}^{J}\wbr{E}_{ij}(t)+\sum_{j=1}^{J}\wbr{D}_{ij}(t)
+\wbr{R}_{i}(t),
\end{equation}
where
$\wbr{R}_{i}(t)=
\int_{0}^{t} \ind{\wbr{Q}_{i}(s)=K_{i}}   d\wbr{R}_{i}(s)=
\int_{0}^{t} \ind{K_{i}-\wbr{Q}_{i}(s)=0}   d\wbr{R}_{i}(s)$.
Now, by the one-dimensional reflection mapping \cite[Chapter~6]{chen2001fundamentals}, we have that
\begin{equation}\label{ch5:key2}
K_{i}-\wbr{Q}_{i}(t)=\Psi(\Phi_i)(t):=\Phi_i(t)
+\sup_{0\leq s\leq t }(-\Phi_i(s)\Max 0),
\end{equation}
where
\begin{equation*}
\Phi_i(t):=
K_{i}-\wbr{Q}_{i}(0)
-\sum_{j=1}^{J}\wbr{E}_{ij}(t)+\sum_{j=1}^{J}\wbr{D}_{ij}(t).
\end{equation*}
It is known that the reflection mapping $\Psi(\cdot)$ is Lipschitz continuous \cite{chen2001fundamentals}.
Now, for each $i\geq 1$, define the mapping $B_i$ for each function $a(\cdot)$ on $[0,\infty)$,
\begin{equation*}
\begin{split}
B_i(a)(t)=\zeta_i(t)-
\sum_{j=1}^{J}\int_{0}^{t}\frac{\lambda_{ij}(s)}
{\sum_{j=1}^{J}\lambda_{ij}(s)}
a(s)f_{D_{ij}}(t-s)ds\\+
\sum_{j=1}^{J}\int_{0}^{t}\int_{0}^{s} a(u)
d\frac{\lambda_{ij}(u)}{\sum_{j=1}^{J}\lambda_{ij}(u)}
f_{D_{ij}}(t-s) ds,
\end{split}
\end{equation*}
where
\begin{equation*}
\zeta_i(t)=K_i-\wbr{Q}_{i}(0)+
\sum_{j=1}^{J}\wbr{Q}_{ij}(0) \Prob{ D^0_{ij} < t }-\sum_{j=1}^{J}\wbr{E}_{ij}(t)+
\sum_{j=1}^{J}\int_{0}^{t}\wbr{E}_{ij}(u)f_{D_{ij}}(t-u)du.
\end{equation*}
Observing that $\frac{\lambda_{ij}(\cdot)}{\sum_{h=1}^{J}\lambda_{ih}(\cdot)}\leq 1$, we have that the mapping $B_i(\cdot)$ is locally Lipschitz continuous for any $i\geq 1$, namely
\begin{equation*}
\sup_{0\leq t \leq T}|B_i(a_1)(t)- B_i(a_2)(t)|
\leq 2\sum_{j=1}^{J}\Prob{D_{ij}\leq T}
\sup_{0\leq t \leq T}|a_1(t)- a_2(t)|.
\end{equation*}
By \cite[Lemma~3]{kang2015}, the following functional equation for any $i\geq 1$ has a unique solution on $[0,T]$:
\begin{equation}\label{ch5:functional}
\begin{split}
a(t)=\Psi(B_i(a))(t)-B_i(a)(t).
\end{split}
\end{equation}
The main idea now is to show that each function $R_i(\cdot)$ satisfies \eqref{ch5:functional}, and hence it is unique.
To this end, by the proof of Proposition~\ref{ch5:analysisQ}, the relation $\wbr{R}_{ij}(t)=\int_{0}^{t} \frac{\lambda_{ij}(s)}{\sum_{h=1}^{J}\lambda_{ih}(s)}d\wbr{R}_{i}(s)$, and the properties of the Riemann-Stieltjes integral, we obtain
\begin{align*}
\wbr{D}_{ij}(t)&= \wbr{Q}_{ij}(0) \Prob{ D^0_{ij} < t }+
\int_{0}^{t}  \Prob{ D_{ij} < t-s } d\wbr{E}_{ij}(s)
-\int_{0}^{t}  \Prob{ D_{ij} < t-s } d\wbr{R}_{ij}(s)\\
&=\wbr{Q}_{ij}(0) \Prob{ D^0_{ij} < t }+
\int_{0}^{t} \wbr{E}_{ij}(s) f_{D_{ij}}(t-u) ds
-\int_{0}^{t} \wbr{R}_{ij}(s) f_{D_{ij}}(t-u) ds
\end{align*}
and
\begin{align*}
\int_{0}^{t} \wbr{R}_{ij}(s) f_{D_{ij}}(t-u) ds&=
\int_{0}^{t} \int_{0}^{s}
\frac{\lambda_{ij}(u)}{\sum_{h=1}^{J}\lambda_{ih}(u)}
d\wbr{R}_{i}(u)
f_{D_{ij}}(t-s)ds\\
&=\int_{0}^{t}\frac{\lambda_{ij}(s)}
{\sum_{h=1}^{J}\lambda_{ih}(s)}
\wbr{R}_{i}(s) f_{D_{ij}}(t-s)ds\\
&\qquad -
\int_{0}^{t}\int_{0}^{s} \wbr{R}_{i}(u)
d\frac{\lambda_{ij}(u)}{\sum_{h=1}^{J}\lambda_{ih}(u)}
f_{D_{ij}}(t-s) ds.
\end{align*}
Using the last equation and replacing $\wbr{D}_{ij}(t)$ in \eqref{ch5:key}, we have that
\begin{align*}
K_{i}-\wbr{Q}_{i}(t)=& \zeta_i(t)
-\sum_{j=1}^{J}\int_{0}^{t}\frac{\lambda_{ij}(s)}
{\sum_{h=1}^{J}\lambda_{ih}(s)}
\wbr{R}_{i}(s) f_{D_{ij}}(t-s)ds\\
&+
\sum_{j=1}^{J}\int_{0}^{t}\int_{0}^{s} \wbr{R}_{i}(u)
d\frac{\lambda_{ij}(u)}{\sum_{h=1}^{J}\lambda_{ih}(u)}
f_{D_{ij}}(t-s) ds + \wbr{R}_{i}(t)\\
=& B_i(\wbr{R}_{i})(t)+\wbr{R}_{i}(t).
\end{align*}
Using again the reflection mapping, we obtain
\begin{equation*}
 K_{i}-\wbr{Q}_{i}(t)=\Psi(B_i(\wbr{R}_{i}))(t).
\end{equation*}
The last equation and \eqref{ch5:key2} yield
\begin{equation}\label{ch5:key3}
 \Phi_i(t)=B_i(\wbr{R}_{i})(t).
\end{equation}
Combining \eqref{ch5:key} and \eqref{ch5:key2}, we derive
\begin{equation*}
 \wbr{R}_{i}(t)=\Psi(\Phi_i)(t)-\Phi_i(t).
\end{equation*}
Now, replacing $\Phi_i(\cdot)$  in the last equation by the right hand side of \eqref{ch5:key3} leads to
\begin{equation*}
 \wbr{R}_{i}(t)=\Psi(B_i(\wbr{R}_{i}))(t)-B_i(\wbr{R}_{i})(t).
\end{equation*}
Thus, $\wbr{R}_{i}(\cdot)$ is a solution of \eqref{ch5:functional}, and hence unique. This implies that
$\wbr{R}_{ij}(\cdot)$ is unique for any $i,j\geq 1$ and hence, $(\wbr{\Q}_{ij}(\cdot),\wbr{Q}_{ij}(\cdot))$ is unique for $i,j\geq 1$.

We now proceed to show the uniqueness of the $\wbr{Z}_{ij}(\cdot)$. First, we show that
$\wbr{\Z}_{ij}(\cdot)$ has a Lipschitz continuous first projection. Indeed,
let $x<x'$ and $y\geq 0$. For any $i,j\geq 0$, we have that
\begin{align*}
 \wbr{\Z}_{ij}(t)\left(
 [x,x']\times [y,\infty)
 \right)
 \leq
 \wbr{\Z}_{ij}(0)\left(
 [x+S_{ij}(\blt{Z},0,t),x'+S_{ij}(\blt{Z},0,t)]\times [y,\infty)\right)\\
 + \int_{0}^{t} \Prob{x+S_{ij}(\blt{Z},s,t)\leq B_{ij}
 \leq x'+S_{ij}(\blt{Z},s,t)}d\wbr{E}_{ij}(s).
\end{align*}
By the Lipschitz continuity of $\wbr{E}_{ij}(\cdot)$, the previous bound becomes
\begin{align*}
 \wbr{\Z}_{ij}(t)\left(
 [x,x']\times [y,\infty)
 \right)
 \leq
 \wbr{\Z}_{ij}(0)\left(
 [x+S_{ij}(\blt{Z},0,t),x'+S_{ij}(\blt{Z},0,t)]\times [y,\infty)\right)\\
 + \eta_{ij} \int_{0}^{t} \Prob{x+S_{ij}(\blt{Z},s,t)\leq B_{ij}
 \leq x'+S_{ij}(\blt{Z},s,t)}ds.
\end{align*}
By the assumption of the Lipschitz continuity  of the initial condition, the change of variable
$v=\Theta(s)=S_{ij}(\blt{Z},s,t)$, and \cite[Lemma~5]{remerova14}, we have that
\begin{align*}
 \wbr{\Z}_{ij}(t)\left(
 [x,x']\times [y,\infty)
 \right)
 &\leq
 L(x'-x)
 + \eta_{ij} \int_{0}^{S_{ij}(\blt{Z},s,t)}
 \frac{\Prob{x+v\leq B_{ij}
 \leq x'+v}}
 {p_{ij}\left(\blt{Z}(\Theta^{-1}(s))\right)} ds\\
 &\leq \left(L+\|\blt{\eta}\| \sup_{0\leq s\leq t}
 \frac{1}{\blt{Z}(s)}\right)(x'-x).
\end{align*}
That is, the first projection of $\wbr{\Z}_{ij}(\cdot)$ is Lipschitz continuous with constant
\begin{equation*}
L+\|\blt{\eta}\| \sup_{0\leq s\leq t}
 \frac{1}{\blt{Z}(s)}< \infty,
\end{equation*}
where the last inequality follows by Theorem~\ref{Ch5Pr:continuous}. Note now that point $\blt{0}$ is a feasible point of \eqref{OP}. Further, for a vector $\blt{z}$ such that $z_{ij}$ is small enough the power flow constraints are satisfied and hence $p_{ij}(\blt{z})=c_j^{max}$. Moreover, the power allocation function is Lipschitz continuous since we consider the linearized Distflow power flow model as we discussed in Section~\ref{ch5:Perturbation}.
Now, by the Lipschitz continuity of $\wbr{E}_{ij}(\cdot)$ and by applying \cite[Theorem~1]{remerova14}, we obtain that the fluid model solution $(\blt{\Z}(\cdot), \blt{Z}(\cdot))$ is unique.
 \end{proof}

\section{Proof of fluid limit Theorem~\ref{Ch5thm:fluid limit}
}
\label{Ch5:proof fluid lilit}
\subsection{Establishing tightness}
The first step of the proof of Theorem~\ref{Ch5thm:fluid limit} is to show that $\left(\blt{\wbr{\Q}}^n(\cdot),
\blt{\wbr{\Z}}^n(\cdot)\right)$ is C-tight, i.e., tight with continuous weak limits. To do so, we follow the idea of proof of \cite[Theorem~5]{remerova14}.
First, we show that both processes satisfy the compact containment property.
To this end, note that the following bounds hold almost surely
\begin{equation}\label{eq:BQ}
\Q_{ij}(t)\leq\sum_{l=1}^{Q_{ij}(0)}
\delta^{+}_{D_{ijl}^0(t) }
+\sum_{l=1}^{E_{ij}(t)} \delta^{+}_{ D_{ijl}(t) }
\end{equation}
and
\begin{equation}\label{eq:BZ}
\Z_{ij}(t)\leq\sum_{l=1}^{Z_{ij}(0)}
\delta^{+}_{\left(B_{ijl}^0(t), D_{ijl}^0(t) \right)}
+\sum_{l=1}^{E_{ij}(t)} \delta^{+}_{\left(B_{ijl}(t), D_{ijl}(t) \right)}.
\end{equation}
Moreover, by our assumptions,
$\frac{E_{ij}^n(\cdot)}{n}\overset{d}\rightarrow \wbr{E}_{ij}(\cdot)$.
Hence, all the bounds
in \cite[Lemma~9]{remerova14} hold true for the measure-valued processes  $\Q_{ij}(\cdot)$ and $\Z_{ij}(\cdot)$.
That is, for any $T>0$ and $\epsilon>0$, there exist  compact sets
$C\in \mathcal{M}(\R_+)^{I\times J}$ and
$C'\in \mathcal{M}(\R_+^2)^{I\times J}$ such that
\begin{equation}\label{compactQ}
 \liminf_{n \rightarrow \infty}
\Probn{\blt{\wbr{\Q}}^n(t) \in C\ \forall\ t\in [0,T]}
\geq 1-\epsilon,
\end{equation}
and
\begin{equation}\label{compactZ}
 \liminf_{n \rightarrow \infty}
\Probn{\blt{\wbr{\Z}}^n(t) \in C'\ \forall\ t\in [0,T]}
\geq 1-\epsilon.
\end{equation}

Next, we shall show the oscillation control. To do so, we first show a preliminary result. Define $H^b_{a}:=\R_+\times [a,b]$ and $V^b_a:=[a,b] \times \R_+$. If $b=\infty$, then $H^{\infty}_{a}:=\R_+\times [a,\infty)$ and $V^b_{a}:=[a,\infty) \times \R_+$.

\begin{proposition}\label{prop:boundary}
For any $T>0$, $\delta>0$, and $\epsilon>0$, there exist $\alpha>0$ and $b>0$ such that
\begin{equation*}
 \liminf_{n \rightarrow \infty}
\Probn{\sup_{0\leq t\leq T}\sup_{x\in \R_+}\left(
||\blt{\wbr{\Q}}^n(t)([x,x+\alpha])|| \right)
\leq \delta}
\geq 1-\epsilon
\end{equation*}
and
\begin{equation*}
 \liminf_{n \rightarrow \infty}
\Probn{\sup_{0\leq t\leq T}\sup_{x\in \R_+}\left(
||\blt{\wbr{\Z}}^n(t)(H^{x+b}_x)|| \Max
||\blt{\wbr{\Z}}^n(t)(V^{x+b}_x)|| \right)
\leq \delta}
\geq 1-\epsilon.
\end{equation*}
\end{proposition}
\begin{proof}
By \cite[Lemma~10]{remerova14}, we have that there exist $\alpha>0$ and $b>0$ such that
\begin{equation}\label{eqQ}
 \liminf_{n \rightarrow \infty}
\Probn{\sup_{x\in \R_+}\left(
||\blt{\wbr{\Q}}^n(0)([x,x+\alpha])||\right)
\leq \delta}
\geq 1-\epsilon
\end{equation}
and
\begin{equation}\label{eqZ}
 \liminf_{n \rightarrow \infty}
\Probn{\sup_{x\in \R_+}\left(
||\blt{\wbr{\Z}}^n(0)(H^{x+b}_x)|| \Max
||\blt{\wbr{\Z}}^n(0)(V^{x+b}_x)|| \right)
\leq \delta}
\geq 1-\epsilon.
\end{equation}
Next, define
\begin{equation*}
\Q_{ij}^{\infty}(t):=\sum_{l=1}^{E_{ij}(t)} \delta^{+}_{ D_{ijl}(t) },\
\Z_{ij}^{\infty}(t):=\sum_{l=1}^{E_{ij}(t)} \delta^{+}_{\left(B_{ijl}(t), D_{ijl}(t) \right)}.
\end{equation*}
We shall show that
\begin{equation}\label{eq:Qbound}
 \liminf_{n \rightarrow \infty}
\Probn{\sup_{0\leq t\leq T}\sup_{x\in \R_+}\left(
||\blt{\wbr{\Q}}^{n,\infty}(t)([x,x+\alpha])|| \right)
\leq \delta}
\geq 1-\epsilon
\end{equation}
and
\begin{equation}\label{eq:Zbound}
 \liminf_{n \rightarrow \infty}
\Probn{\sup_{0\leq t\leq T}\sup_{x\in \R_+}\left(
||\blt{\wbr{\Z}}^{n,\infty}(t)(H^{x+b}_x)|| \Max
||\blt{\wbr{\Z}}^{n,\infty}(t)(V^{x+b}_x)|| \right)
\leq \delta}
\geq 1-\epsilon.
\end{equation}
Then, the result follows. Indeed,
by \eqref{eq:BQ}, \eqref{eq:BZ}, we have that
\begin{equation*}
\begin{split}
\liminf_{n \rightarrow \infty}\ &
\Probn{\sup_{0\leq t\leq T}\sup_{x\in \R_+}\left(
\|\blt{\wbr{\Q}}^{n}(t)([x,x+\alpha])\| \right)
\leq \delta}\\
&\geq
 \liminf_{n \rightarrow \infty}
\Probn{\sup_{0\leq t\leq T}\sup_{x\in \R_+}\left(
\|\blt{\wbr{\Q}}^{n,\infty}(t)([x,x+\alpha])\| \right)
\leq \delta}
\end{split}
\end{equation*}
and
\begin{equation*}
\begin{split}
\liminf_{n \rightarrow \infty}\ &
\Probn{\sup_{0\leq t\leq T}\sup_{x\in \R_+}\left(
\|\blt{\wbr{\Z}}^{n}(t)(H^{x+b}_x)\| \Max
\|\blt{\wbr{\Z}}^{n}(t)(V^{x+b}_x)\| \right)
\leq \delta}
\\ \geq
& \liminf_{n \rightarrow \infty}
\Probn{\sup_{0\leq t\leq T}\sup_{x\in \R_+}\left(
\|\blt{\wbr{\Z}}^{n,\infty}(t)(H^{x+b}_x)\| \Max
\|\blt{\wbr{\Z}}^{n,\infty}(t)(V^{x+b}_x)\| \right)
\leq \delta}.
\end{split}
\end{equation*}
Now, Proposition~\ref{prop:boundary} follows by using the last inequalities,
\eqref{eqQ}--\eqref{eq:Zbound},
and
\cite[Lemma~12]{remerova14}.

We move now to the proof of \eqref{eq:Qbound} and \eqref{eq:Zbound}.
 Denote by $\Omega_{0,q}^{n}$ and $\Omega_{0,z}^n$ the events for which \eqref{eq:Qbound} and \eqref{eq:Zbound} hold, respectively.
Let $\Omega_{1,q}^{n}$ and $\Omega_{1,z}^n$ be the events for which \eqref{compactQ} and \eqref{compactZ} hold, respectively. By \cite[Proposition~1]{remerova14}, $C$ and $C'$ are relatively compact. Hence,
$\Xi:=\sup\limits_{\blt{m}\in C}\|\blt{m}(\R_+)\| < \infty$,
$\Xi':=\sup\limits_{\blt{m}\in C'}\|\blt{m}(\R_+^2)\|< \infty$,
$\sup\limits_{\blt{m}\in C}\|\blt{m}(\R_+\setminus[0,L])\|\leq \delta/4$, and
$\sup\limits_{\blt{m}\in C'}\|\blt{m}(\R_+^2\setminus[0,L']^2)\|\leq \delta/4$ for large $L$ and $L'$. In addition, put
$p_*:=\min\limits_{i,j}\{p_{ij}: z_{ij}>\delta/4, ||\blt z||\leq \Xi'\}$,
$\beta:=\frac{\delta}{8\|\blt{\eta}\|}\Min T$, $\alpha=\frac{\beta}{3}$, and
$b=\frac{\beta(p_*\Min 1)}{3}$.
Further, take $N$ and $N'$ such that
\begin{align*}
 N \alpha >  L+T \text{ and }
 N' b> L' + (\|\blt{c}^{max}\|\Max 1) T,
\end{align*}
and define the following sets
\begin{align*}
  I_{k}:=&[(k-1)\alpha, k\alpha],  \\
  I^{k}:=& [(k-2)^+\alpha, (k+1)\alpha], \\
  I_{k,k'}:=& [(k-1)b, kb]\times[(k'-1)b, k'b], \\
  I^{k,k'}:=& [(k-2)^+b, (k+1)b]\times[(k'-2)^+b, (k'+1)b].
\end{align*}
Furthermore, pick functions $g_k \in \blt{C}(\R_+,[0,1])$  and
$g_{k,k'} \in \blt{C}(\R_+^2,[0,1])$ such that
\begin{align*}
\ind{I_{k}}(\cdot) &\leq g_{k}(\cdot) \leq \ind{I^{k}}(\cdot),\\
\ind{I_{k,k'}}(\cdot) & \leq g_{k,k'}(\cdot) \leq \ind{I^{k,k'}}(\cdot),
\end{align*}
and note that
\begin{align*}
  \sum_{k\in \mathbb{N}}||<g_{k},\blt{F_D}>|| & \leq
  || \sum_{k\in \mathbb{N}}<g_{k},\blt{F_D}>|| \leq 3,  \\
    \sum_{k,k'\in \mathbb{N}}||<g_{k,k'},\blt{F}>|| & \leq
  || \sum_{k,k'\in \mathbb{N}}<g_{k,k'},\blt{F}>|| \leq 9.
\end{align*}
Define the load processes for the $n^{\text{th}}$ system, and $t\geq 0$,
\begin{equation*}
\mathcal{L}^{n,Q}_{ij}(t):=\sum_{l=1}^{E_{ij}^n(t)}\delta_{D_{ijl}},\hspace{1cm}
\mathcal{L}^{n,Z}_{ij}(t):=\sum_{l=1}^{E_{ij}^n(t)}
\delta_{\left(B_{ijl},D_{ijl}\right)},
\end{equation*}
and the corresponding  scaled load processes
\begin{equation*}
\wbr{\mathcal{L}}^{n,Q}_{ij}(t):=\frac{\mathcal{L}^{n,Q}_{ij}(nt)}{n},\hspace{1cm}
\wbr{\mathcal{L}}^{n,Z}_{ij}(t):=\frac{\mathcal{L}^{n,Z}_{ij}(nt)}{n}.
\end{equation*}
By \cite[Theorem~5.1]{gromoll2009}, we have that
\begin{align*}
\lim_{n \rightarrow \infty}
 \Probn{ \max_{1 \leq k \leq N} \sup_{0 \leq t \leq T}
||<g_{k},\wbr{\blt{\mathcal{L}}}^{n,Q}>-
\wbr{\blt{E}}(t) <g_{k},\blt{F_D}> || \leq \frac{\delta}{16 N^2}}=1
\end{align*}
and
\begin{align*}
\lim_{n \rightarrow \infty}
\Probn{ \max_{1 \leq k,k' \leq N'} \sup_{0 \leq t \leq T}
||<g_{k,k'},\wbr{\blt{\mathcal{L}}}^{n,Z}>-
\wbr{\blt{E}}(t) <g_{k,k'},\blt{F}> ||
\leq \frac{\delta}{16 N'^2}}=1,
\end{align*}
where we denote by $\Omega_{2,q}^{n}$ and $\Omega_{2,z}^n$ the corresponding events.
Further, by our assumptions
\begin{align*}
\lim_{n \rightarrow \infty}
\Probn{  \sup_{0 \leq t \leq T}
||\wbr{\blt{E}}^n(t)-
\wbr{\blt{E}}(t)||\leq \delta/16}=1,
\end{align*}
and denote these events by  $\Omega_{3}^{n}$.
Adapting the proof of  \cite[Lemma~11]{remerova14}, it follows that
$
\Omega_{1,q}^{n} \cap \Omega_{2,q}^{n} \cap\Omega_{3}^{n} \subseteq
 \Omega_{0,q}^{n}$
and
$\Omega_{1,z}^{n} \cap \Omega_{2,z}^{n} \cap\Omega_{3}^{n}
\subseteq \Omega_{0,z}^{n}$. This concludes the proof of Proposition~\ref{prop:boundary}.
\end{proof}

\begin{proposition}[Oscillation control]
For any $T>0$, $\delta>0$, and $\epsilon>0$ there exist $h>0$ and $h'>0$ such that
\begin{equation}\label{eq:oscQ}
\liminf_{n \rightarrow \infty}
\Probn{\omega(
\blt{\wbr{\Q}}^{n}(\cdot),h,T)
\leq \delta}
\geq 1-\epsilon
\end{equation}
and
\begin{equation}\label{eq:oscZ}
 \liminf_{n \rightarrow \infty}
\Probn{\omega(
\blt{\wbr{\Z}}^{n}(\cdot),h',T)
\leq \delta}
\geq 1-\epsilon,
\end{equation}
where for a measure-valued process $\X(\cdot)$ we define
\begin{equation*}
\omega(
\blt{\X}(\cdot),h,T)
:=\sup\limits_{0\leq s,t \leq T}\{
\blt{d}(\blt{\X}(t),\blt{\X}(s) ): |t-s|<h\}.
\end{equation*}
\end{proposition}
\begin{proof}
We shall use the idea of proof of \cite[Lemma~13]{remerova14}. Let $\Omega_q^n$ and $\Omega_z^n$ be the events such that \eqref{eq:oscQ} and \eqref{eq:oscZ} hold, respectively.
Denote by $\Omega_1^n$ the following events
\begin{equation*}
\lim_{n \rightarrow \infty}
\Probn{  \sup_{0 \leq t \leq T}
||\wbr{\blt{E}}^n(t)-
\wbr{\blt{E}}(t)||\leq \delta/4}=1.
\end{equation*}
Further, by Proposition~\ref{prop:boundary}, there exist $a>0$ and $b>0$ such that
\begin{equation*}
 \liminf_{n \rightarrow \infty}
\Probn{\sup_{0\leq t\leq T}
||\blt{\wbr{\Q}}^n(t)([0,\alpha])||
\leq \delta}
\geq 1-\epsilon
\end{equation*}
and
\begin{equation*}
 \liminf_{n \rightarrow \infty}
\Probn{\sup_{0\leq t\leq T}
||\blt{\wbr{\Z}}^n(t)(H^{b}_0\cup V^{b}_0)||
\leq \delta}
\geq 1-\epsilon.
\end{equation*}
Denote the corresponding events by $\Omega_{2,q}^n$ and $\Omega_{2,z}^n$, respectively. Now, choose $h$ and $h'$ such that $h||\eta||\leq \delta/2$, $h\leq \delta\Max \alpha$ and
$h'(||c^{max}||\Max 1)\leq \delta\Max b$, $h'||\eta||\leq \delta/2$.

We shall show that $\Omega_{1}^n\cap \Omega_{2,q}^n\subseteq\Omega_{q}^n$
and $\Omega_{1}^n\cap \Omega_{2,z}^n\subseteq\Omega_{z}^n$. Take $0\leq s<t\leq T$ with $t-s<h$. Let $\omega\in \Omega_{1}^n\cap \Omega_{2,q}^n $, we shall show that for any non-empty closed Borel set
$B\subseteq \R_+$,
\begin{align}
\wbr{\Q}_{ij}^n(s)(B)\leq \wbr{\Q}_{ij}^n(t)(B^\delta)+\delta, \label{eq:oscQ1}\\
\wbr{\Q}_{ij}^n(t)(B)\leq \wbr{\Q}_{ij}^n(s)(B^\delta)+\delta, \label{eq:oscQ2}
\end{align}
where $B^\delta:=\{x\in \R_+: \inf\limits_{y\in B}
||x-y||\leq \delta\}$. Then
\eqref{eq:oscQ} follows.
First, we prove \eqref{eq:oscQ1}. Define $\tau:=\inf\{s\leq u\leq t: \wbr{Q}^n_{ij}(u)=0\}\Min t$. Then, we have that
\begin{equation*}
\wbr{\Q}^n_{ij}(s)(B)\leq \wbr{\Q}^n_{ij}(s)(B\cap [\alpha,\infty))+
\wbr{\Q}^n_{ij}(s)([0,\alpha))
\leq
\wbr{\Q}^n_{ij}(s)(B\cap [\alpha,\infty))+\delta,
\end{equation*}
where the last inequality holds because
$\omega\in \Omega_{2,q}^n $. Now, observe that
\begin{equation*}
\wbr{\Q}^n_{ij}(s)(B\cap [\alpha,\infty))
\leq
\wbr{\Q}^n_{ij}(\tau)(B^{\delta}),
\end{equation*}
because $\tau-s<h<\delta\Min \alpha$. To see the last statement observe that if for some EV in the system at time $s$, $D_{ijl}-(s-\zeta_{ijl})\in B$ then
$D_{ijl}-(s-\zeta_{ijl})-D_{ijl}+(\tau-\zeta_{ijl})\leq \delta$, which yields
$D_{ijl}+(\tau-\zeta_{ijl})\in B^{\delta}$. Finally, we have that
\begin{equation*}
\wbr{\Q}^n_{ij}(s)(B)\leq
\wbr{\Q}^n_{ij}(\tau)(B^{\delta})+\delta.
\end{equation*}
Now, if $\tau=t$, then \eqref{eq:oscQ1} follows. If $\tau< t$, then
$0\leq \wbr{\Q}^n_{ij}(\tau)(B^{\delta})\leq \wbr{Q}^n_{ij}(\tau)=0$, and \eqref{eq:oscQ1} follows. To show \eqref{eq:oscQ2}, we write
\begin{equation*}
\begin{split}
\wbr{\Q}^n_{ij}(t)(B)&\leq
\wbr{\Q}^n_{ij}(s)(B^{\delta})
+\wbr{E}^n_{ij}(t)-\wbr{E}^n_{ij}(s)+
\wbr{R}^n_{ij}(s)-\wbr{R}^n_{ij}(t)\\
&\leq
\wbr{\Q}^n_{ij}(s)(B^{\delta})+\wbr{E}^n_{ij}(t)
-\wbr{E}^n_{ij}(s),
\end{split}
\end{equation*}
where the second inequality follows because $\wbr{R}^n_{ij}(s)-\wbr{R}^n_{ij}(t)\leq 0$. Now, \eqref{eq:oscQ2} follows because $\omega\in \Omega_{1}^n $. We conclude that $\omega\in \Omega_{q}^n $. The proof of $\Omega_{1}^n\cap \Omega_{2,z}^n\subseteq\Omega_{z}^n$ follows by similar arguments.
\end{proof}

\subsection{Fluid limits satisfy the fluid model solutions}
Note that the total number of EVs can be written as follows
\begin{equation}\label{eq:QL}
\wbr{Q}^n_{ij}(t)=\wbr{Q}^n_{ij}(0)+\wbr{E}^n_{ij}(t)-\wbr{R}^n_{ij}(t)
-\wbr{D}^n_{ij}(t),
\end{equation}
where the number of rejected EVs $\wbr{R}^n_{ij}(\cdot)$ is given by \eqref{eq:Rejected} and
\begin{equation*}
  \wbr D^n_{ij}(t):= \frac{1}{n}\sum_{l=1}^{n\wbr Q^n_{ij}(0)}
  \ind{ D^0_{ijl}\leq t}
 +\frac{1}{n}\sum_{l=1}^{n \wbr E^n_{ij}(t)}
 \ind{\zeta_{ijl}+ D_{ijl}\leq t}
 \ind{\wbr Q^n_{i}(\zeta_{ijl}^{-})<K}.
\end{equation*}

\begin{proposition}\label{Pr:tightness}
The fluid-scaled stochastic processes
$\wbr{\blt{D}}^n(\cdot)$ and $\wbr{\blt{R}}^n(\cdot)$
are tight.
\end{proposition}
\begin{proof}
First, we shall show that $\wbr D^n_{ij}(\cdot)$ is a relatively compact sequence using Kurtz's criteria (see \cite[Proposition~6.2]{kang2010fluid}), then by Prokhorov's Theorem, it is tight.
Observe that almost surely
\begin{equation*}
 \wbr D^n_{ij}(t)\leq \frac{1}{n}\sum_{l=1}^{n\wbr Q^n_{ij}(0)}
 \ind{ D^0_{ijl}\leq t}
 +\frac{1}{n}\sum_{l=1}^{n \wbr E^n_{ij}(t)}
  \ind{\zeta_{ijl}+ D_{ijl}\leq t}=:
 \wbr D^{n,\infty}_{ij}(t),
\end{equation*}
and by \cite{reed2009} the latter is a weakly convergent sequence in $(\mathcal{D}[0,\infty),J_1)$  and hence it is tight. By Prokhorov's Theorem it is also relatively compact.
That is,
\begin{align*}
   \lim_{c\rightarrow \infty} \Prob{\wbr D^n_{ij}(t)>c} \leq
   \lim_{c\rightarrow \infty} \Prob{\wbr D^{n,\infty}_{ij}(t)>c}=0.
\end{align*}
In other words, $\wbr D^n_{ij}(\cdot)$ is stochastically bounded and hence satisfies the first property of Kurtz's criteria.
To show that it also satisfies the second property,  we write
\begin{align*}
 \wbr D^n_{ij}(t+\delta)- \wbr D^n_{ij}(t&) =
 \wbr D^{n,\infty}_{ij}(t+\delta)- \wbr D^{n,\infty}_{ij}(t)
 +\frac{1}{n}
 \sum_{l=1}^{n \wbr E^n_{ij}(t+\delta)} \ind{\zeta_{ijl}+ D_{ijl}
 \leq t+\delta}
 \ind{\wbr Q^n_i(\zeta_{ijl}^{-})=K_i}\\
& \hspace{4cm} -
 \frac{1}{n}\sum_{l=1}^{n \wbr E^n_{ij}(t)} \ind{\zeta_{ijl}+ D_{ijl}\leq t}
 \ind{\wbr Q^n_i(\zeta_{ijl}^{-})=K_i}
 \\
 &=\wbr D^{n,\infty}_{ij}(t+\delta)- \wbr D^{n,\infty}_{ij}(t)+
 \sum_{l=1}^{n \wbr E^n_{ij}(t+\delta)} \ind{t<\zeta_{ijl}+
  D_{ijl}\leq t+\delta}
 \ind{\wbr Q^n_i(\zeta_{ijl}^{-})=K_i}.
\end{align*}
Note that for any $t\geq 0$ and $n\geq 1$,
\begin{equation*}
\begin{split}
 \frac{1}{n}\sum_{l=1}^{n \wbr E^n_{ij}(t+\delta)}
  \ind{t<\zeta_{ijl}+ D_{ijl}\leq t+\delta}
 \ind{\wbr Q^n_i(\zeta_{ijl}^{-})=K_i}
 &\leq
 \frac{1}{n} \sum_{l=1}^{n \wbr E^n_{ij}(t+\delta)}
  \ind{t<\zeta_{ijl}+ D_{ijl}\leq t+\delta}\\
  &\leq
\sup_{n} \wbr E^n_{ij}(t+\delta)< \infty .
\end{split}
\end{equation*}
Further, by continuity of the random variables $\zeta_{ijl}$ and $D_{ijl}$, we have that as $\delta \rightarrow 0$,
\begin{equation}\label{eq:cont}
   \sum_{l=1}^{n \wbr E^n_{ij}(t+\delta)} \ind{t<\zeta_{ijl}+ D_{ijl}\leq t+\delta}
   \rightarrow 0.
\end{equation}
Putting all the pieces together,
\begin{align*}
 |\wbr D^n_{ij}(t+\delta)- \wbr D^n_{ij}(t)|& \leq
 |\wbr D^{n,\infty}_{ij}(t+\delta)- \wbr D^{n,\infty}_{ij}(t)|
+\frac{1}{n}\sum_{l=1}^{n \wbr E^n_{ij}(t+\delta)}
\ind{t<\zeta_{ijl}+ D_{ijl}\leq t+\delta}.
\end{align*}
By \eqref{eq:cont}, the fact that $\wbr D^{n,\infty}_{ij}(\cdot)$ is relatively compact and using the same arguments as in
\cite[Lemma~5.10]{kaspi2011}, we conclude that $ \wbr D^n_{ij}(\cdot)$ satisfies the second property of Kurtz's criteria. That is,
$\wbr D^n_{ij}(\cdot)$
 is relatively compact and hence tight. The tightness of $\wbr{\blt{R}}^n(\cdot)$ follows by \eqref{eq:QL} and by the tightness of $\wbr{\blt{D}}^n(\cdot)$ and $\wbr{\blt{Q}}^n(\cdot)$.
\end{proof}
Next, we show that the fluid limits are bounded away from zero.
\begin{proposition}
Let $(\wbr{\blt{\Q}}(\cdot),\wbr{\blt{Q}}(\cdot),\wbr{\blt{\Z}}(\cdot),\wbr{\blt{Z}}(\cdot))$.
be a fluid limit.
Assume that if\ $\wbr{Q}_i(0)=K_i$, then $0<\wbr{Q}_{ij}(0)<K_i$ for any $i,j\geq 1$. For any $\delta >0$, there exist $C_{\delta}>0$
and  $C_{\delta}'>0$
such that
almost surely
\begin{equation*}
 \inf_{t\geq \delta}\min_{i,j} \wbr{Q}_{ij}(t)\geq C_{\delta}
 \ \text{ and }\
 \inf_{t\geq \delta}\min_{i,j} \wbr{Z}_{ij}(t)\geq C_{\delta}'.
\end{equation*}
\end{proposition}
\begin{proof}
First, we shall show that $\wbr{Q}_{ij}(\cdot)$ is strictly positive. It is enough to show this inequality when the system is not full. Fix $\Delta>\delta$. It is enough to show the result for $t\in[\delta,\Delta]$.
Define
\begin{equation*}
  \begin{split}
  \tau_{i}^0:=\inf\{\delta \leq s\leq \Delta: \wbr{Q}_{i}(s)=K_i\},\
  \tilde{\tau}_{i}^0:=\inf\{\tau_i^0\leq s\leq \Delta: \wbr{Q}_{i}(s)<K_i\},\\
  \tau_{i}^r:=\inf\{\tilde{\tau}^{r-1} \leq s\leq \Delta:\wbr{Q}_{i}(s)=K_i\},\
  \tilde{\tau}_{i}^r:=\inf\{\tau_i^r\leq s\leq \Delta: \wbr{Q}_{i}(s)<K_i\}.
  \end{split}
\end{equation*}
Take a partition
\begin{equation*}
(0,\Delta]\setminus
\bigcup\limits_{r}[\tau_{i}^r,\tilde{\tau}_{i}^{r})
\subseteq
\bigcup\limits_{1\leq m\leq N(\Delta)} ((m-1)b/2,mb/2].
\end{equation*}
By our assumptions for the external arrival process, we have that for any $m$,
\begin{equation*}
\frac{1}{n}\sum_{l=E_{ij}^n((m-1)b/2)+1}^{E_{ij}^n(mb/2)}
\ind{D_{ij}\geq b}
\overset{d}\rightarrow
(\wbr{E}_{ij}(mb/2)-\wbr{E}_{ij}((m-1)b/2))\Prob{D_{ij}>b}>0,
\end{equation*}
where $b$ is a continuity point for the distribution $F_{D_{ij}}(\cdot)$ with $\Prob{D_{ij}>b}>0$, and  the last inequality follows because $\wbr{E}_{ij}(\cdot)$ is strictly increasing. Choose $b$ such that
$\max\limits_{ij}(\wbr{E}_{ij}(mb/2)-\wbr{E}_{ij}((m-1)b/2))\Prob{D_{ij}>b}<K_i$, and pick $C_{\delta}$ such that
$\max\limits_{ij}
(\wbr{E}_{ij}(mb/2)-\wbr{E}_{ij}((m-1)b/2))\Prob{D_{ij}>b}>C_{\delta}.$
Then, for large enough $n$, we have that for any $i,j\geq 1$,
\begin{equation*}
\begin{split}
\Probn{
\inf_{\delta \leq t \leq \Delta}
\wbr{Q}_{ij}^n(t)\geq C_{\delta}
}
&\geq
\Probn{\inf_{(m-1)b/2\leq t\leq mb/2} \wbr{Q}_{ij}^n(t)\geq C_\delta\
 \text{for any}\ m}\\
&\geq
\Probn{ \sum_{l=E_{ij}^n((m-1)b/2)+1}^{E_{ij}^n(mb/2)}
\ind{D_{ij}\geq b}\geq C_\delta\
\text{for any}\ m} \rightarrow 1.
 \end{split}
\end{equation*}
Further, note that by continuity of the limit we have
$\wbr{Q}_{ij}^n(t)=\wbr{Q}_{ij}^n(\tau_{i}^r)\geq C_{\delta}$, for $t\in [\tau_{i}^r,\tilde{\tau}_{i}^{r})$. Finally, we have that there exists  $C_{\delta}>0$ such that, for any $\Delta>\delta$,
\begin{align*}
\Probn{
\inf_{\delta \leq t \leq \Delta}
\min_{i,j}
\wbr{Q}^n_{ij}(t)\geq C_{\delta}
} \rightarrow 1,
\end{align*}
as $n\rightarrow \infty$.
For any compact  set $C\subseteq R_+$, define the mapping $\phi_C: D(\R_+,\R^{I\times J})\rightarrow \R$, given by
$\phi_C(\blt{y}):=\inf_{t\in C}\min_{i,j}y_{ij}(t)$. Note that $\phi_C(\blt{y})$ is continuous at continuous $\blt{y}(\cdot)$, which implies that
\begin{equation*}
\phi_{[\delta, \Delta]}
(\wbr{\blt{Q}}^n)
\overset{d}\rightarrow
\phi_{[\delta, \Delta]}
(\wbr{\blt{Q}}).
\end{equation*}
By the Portmanteau theorem \cite[Theorem~2.1]{billingsley1999convergence}, we have that
\begin{align*}
\Probn{
\phi_{[\delta, \Delta]}
(\wbr{\blt{Q}})\geq C_{\delta}
}
\geq
 \limsup_{n\rightarrow \infty}
 \Probn{
\phi_{[\delta, \Delta]}
(\wbr{\blt{Q}}^n)\geq C_{\delta}
} =1,
\end{align*}

We now move to the proof of $\wbr{Z}_{ij}(t)>0$ for $t>0$.  We first note that
$\blt{Q}(\cdot)$ is independent of $\blt{Z}(\cdot)$, and hence we can assume that the fluid limit $(\wbr{\blt{\Q}}(\cdot),\wbr{\blt{Q}}(\cdot))$ satisfies the fluid model equations as we shall show later. That is, $(\wbr{\blt{\Q}}(\cdot),\wbr{\blt{Q}}(\cdot))$ satisfies the equations in Proposition~\ref{ch5:analysisQ}. By
Proposition~\ref{Pr:tightness}, we have that the fluid-scaled process that describes the number of accepted EVs given in \eqref{eq:Accepted} converges weakly to
$\wbr{\blt{A}}(t):=\wbr{\blt{E}}(t)-\wbr{\blt{R}}(t)$. First, we show that
$\wbr{A}_{ij}(t)$ is strictly increasing for any $i,j\geq 1$. Let  $t_1,t_2\geq 0$ with $0\leq t_1< t_2$. Assume that there exists a subinterval in $[t_1,t_2]$ such that the total queue length at node $i$ is full. Without loss of generality, assume that there exists
$\tau\in [t_1,t_2]$ such that $\wbr{Q}_i(s)=K_i$ for any  $s\in [\tau,t_2]$. First, assume that $\tau>t_1$, then we have that
\begin{align*}
\wbr{A}_{ij}(t_2)-\wbr{A}_{ij}(t_1)&=
\wbr{E}_{ij}(t_2)-\wbr{R}_{ij}(t_2)-\wbr{E}_{ij}(t_1)
+\wbr{R}_{ij}(t_1)\\
&\geq \wbr{E}_{ij}(t_2)-\wbr{R}_{ij}(t_2)-\wbr{E}_{ij}(t_1)\geq \wbr{E}_{ij}(\tau)-\wbr{E}(t_1)>0.
\end{align*}
If $\tau=t_1$, then by \eqref{eq:Departures}, \eqref{eq:depE}, and the fact that $\wbr{Q}_{ij}(t_2)=\wbr{Q}_{ij}(t_1)$, we obtain
\begin{align*}
\wbr{A}_{ij}(t_2)-\wbr{A}_{ij}(t_1)&=
\wbr{D}_{ij}(t_2)-\wbr{D}_{ij}(t_1)=\int_{t_1}^{t_2}
\delta_{ij}(s)ds,
\end{align*}
where $\delta_{ij}(s)=\lim_{\epsilon\rightarrow 0}
\frac{\wbr{Q}_{ij}(s)-\wbr{\Q}_{ij}(s)\left([\epsilon,\infty)\right)}
{\epsilon}$. Further, by the proof of Proposition~\ref{ch5prop:positiveZ}, we have that
$\delta_{ij}(s)=\delta_{ij}(t_1)>0$ for $s\in[t_1,t_2]$, and hence $\wbr{A}_{ij}(t_2)-\wbr{A}_{ij}(t_1)>0$. Now, consider a type-$j$ EV $l$ at node $i$. Observe that by the constraints
$p_{ij}(\cdot)\leq c^{\text{max}}_j$, we have that
$\frac{B_{ijl}}{p_{ij}(\cdot)}\Min D_{ijl}\geq
\frac{B_{ijl}}{c^{\text{max}}_j}\Min D_{ijl}$. That is, EV $l$ will stay in the network  at least $\frac{B_{ijl}}{c^{\text{max}}_j}\Min D_{ijl}$ after its arrival. Hence, the stochastic process
$Z_{ij}(\cdot)$ is bounded  from below by the queue length $Q_{ij}^{\text{inf}}(\cdot)$ of the infinite-server queue with arrival process $A_{ij}(\cdot)$, $Q_{ij}^{\text{inf}}(0)=0$, and i.i.d. service requirements
$\{
\frac{B_{ijl}}{c^{\text{max}}_j}\Min D_{ijl}, l\in \mathbb{N} \}$.
Recalling that $\wbr{A}_{ij}(\cdot)$ is strictly increasing by
\cite[Lemma~3.14]{remerova14}, there exists  $C_{\delta}'>0$ such that, for any $\Delta>\delta$,
\begin{align*}
\Probn{
\inf_{\delta \leq t \leq \Delta}
\min_{i,j}
\wbr{Z}^n_{ij}(t)\geq C_{\delta}'
} \geq
\Probn{
\inf_{\delta \leq t \leq \Delta}
\min_{i,j}
\wbr{Q}_{ij}^{n,\text{inf}}(t)\geq C_{\delta}'
} \rightarrow 1,
\end{align*}
as $n\rightarrow \infty$. Now, using again the Portmanteau theorem, we have that
\begin{align*}
\Probn{
\phi_{[\delta, \Delta]}
(\wbr{\blt{Z}})\geq C_{\delta}'
}
\geq
 \limsup_{n\rightarrow \infty}
 \Probn{
\phi_{[\delta, \Delta]}
(\wbr{\blt{Z}}^n)\geq C_{\delta}'
} =1.
\end{align*}

\end{proof}

\subsubsection{Fluid limits are fluid model solutions}
In the sequel, we focus on proving that any fluid limit satisfies the fluid model equations given in Definition~\ref{def:FMSC}. Let $(\wbr{\blt{\Q}}(\cdot),\wbr{\blt{Q}}(\cdot),
\wbr{\blt{\Z}}(\cdot),\wbr{\blt{Z}}(\cdot),\wbr{\blt{R}}(\cdot))$ be a fluid limit along a subsequence, which with an abuse of notation, we denote  again by
 $(\wbr{\blt{\Q}}^n(\cdot),\wbr{\blt{Q}}^n(\cdot),
\wbr{\blt{\Z}}^n(\cdot),\wbr{\blt{Z}}^n(\cdot), \wbr{\blt{R}}^n(\cdot))$.
Recall that
$\mathcal{C}:=\left\{[x,\infty),\ x\in \R_+ \right\}$ and $\mathcal{C'}:=\left\{[x,\infty)\times[y,\infty),\ x,y\in \R_+ \right\}$. Proposition~\ref{prop:boundary} and \cite[Lemma~6.2]{gromoll2008} imply that
for any $\A\in \mathcal{C}$ and $\A'\in \mathcal{C'}$, almost surely
$\wbr{\Q}_{ij}(t)(\partial\A)=0$ and
$\wbr{\Z}_{ij}(t)(\partial\A')=0$ for $t\geq 0$ and $i,j\geq 1$. Hence, we can restrict $\mathcal{C}$ and $\mathcal{C'}$ to the following restricted  classes
$\mathcal{C}_+:=\left\{[x,\infty),\ x>0 \right\}$ and $\mathcal{C'}_+:=\left\{[x,\infty)\times[y,\infty),\ x\Min y>0 \right\}$.
In addition, we fix $T>0$ and we work in the time interval $[0,T]$.

The total number of type-$j$ EVs at node $i$ can be written as follows
\begin{equation*}
\Q_{ij}^n(t)(A)=\Q_{ij}^n(0)\left(A+t\right)
+\sum_{l=1}^{E_{ij}^n(t)}
\indset_A \left(D_{ij}-(t-\zeta_{ijl})\right)\ind{Q_{i}^n(\zeta_{ijl}^-)<K_i}.
\end{equation*}
Further, the above expression can be rewritten as
\begin{equation*}
\Q_{ij}^n(t)(A)=\Q_{ij}^n(0)\left(A+t\right)
+\sum_{l=1}^{A_{ij}^n(t)}
\indset_A\left(D_{ij}-(t-\xi_{ijl})\right),
\end{equation*}
where $\xi_{ijl}$ represents the time of the $l^{th}$ accepted EV and $A_{ij}^n(\cdot)$ represents the number of accepted type-$j$ EVs at node $i$.
In the same way, the number of uncharged type-$j$ EVs at node $i$ is given by
\begin{equation*}
\begin{split}
\Z_{ij}^n(t)(A')=\Z_{ij}^n(0)&\left(A'+(S_{ij}(\blt{Z}^n,0,t),t)\right)
\\&+\sum_{l=1}^{A_{ij}^n(t)}
\indset_{A'}\left(B_{ijl}-S_{ij}(\blt{Z}^n,\xi_{ijl},t),
D_{ij}-(t-\xi_{ijl})\right).
\end{split}
\end{equation*}
In the above expressions, we relabel the parking times and the charging requirements accordingly, where with abuse of notation we denote them by the same letters.
Now, we can follow the strategy in \cite[Section~7.6]{remerova14}.
Consider a partition $0<t_0<\ldots<t_N=t$ and take a nonincreasing function function $y(\cdot)$ in $[t_0,t]$ such that
\begin{equation*}
\sup_{t_0 \leq s \leq t}| S_{ij}(\wbr{\blt{Z}}^n,s,t)-y(s)|\leq \delta,
\end{equation*}
for some $\delta>0$.
We note that for $\xi_{ijl}\in (t_r,t_{r+1}]$, the following inequalities hold
\begin{equation*}
\begin{split}
D_{ijl}-(t-t_r)\leq   D_{ijl}-(t-\xi_{ijl})
 \leq D_{ijl}-(t-t_{r+1}),
  \end{split}
\end{equation*}
\begin{equation*}
  \begin{split}
B_{ijl}-(y(t_r)+\delta)\leq   B_{ijl}-S_{ij}(\blt{Z}^n,\xi_{ijl},t)
 \leq B_{ijl}-(y(t_{r+1})+\delta).
 \end{split}
\end{equation*}
Now, define the following quantities
\begin{equation*}
\begin{split}
\Gamma^{n,1}_{ij}(t):=
\sum_{r=0}^{N-1} \left(\wbr{A}_{ij}(t_{r+1})-\wbr{A}_{ij}(t_{r})\right)
F_{D_{ij}}\left(A+(t-t_{r})\right)-\widetilde{X}^n,
\end{split}
\end{equation*}
\begin{equation*}
\begin{split}
\Gamma^{n,2}_{ij}(t):=\sum_{r=0}^{N-1} \
\left(\wbr{A}_{ij}(t_{r+1})-\wbr{A}_{ij}(t_{r})\right)
F_{D_{ij}}\left(A+(t-t_{r+1})\right)+\widetilde{X}^n
+\wbr{A}_{ij}(t_0)+\widetilde{X}^n,
  \end{split}
\end{equation*}
\begin{equation*}
\begin{split}
\Gamma^{n,3}_{ij}(t):=
\sum_{r=0}^{N-1} \left(\wbr{A}_{ij}(t_{r+1})-\wbr{A}_{ij}(t_{r})\right)
F_{ij}\left(A'+(y(t_r)+\delta,t-t_{r})\right)-X^n,
\end{split}
\end{equation*}
\begin{equation*}
\begin{split}
\Gamma^{n,4}_{ij}(t):=
\sum_{r=0}^{N-1} \left(\wbr{A}_{ij}(t_{r+1})-\wbr{A}_{ij}(t_{r})\right)
F_{ij}\left(A'+(y(t_{r+1})+\delta,t-t_{r+1})\right)\\+X^n
+\wbr{A}_{ij}(t_0)+X^n,
\end{split}
\end{equation*}
where
\begin{equation*}
\widetilde{X}^n:=\sup_{A\in\mathcal{C}}\sup_{0\leq s\leq t\leq T}
\left\|\wbr{\blt{\mathfrak{L}}}^{n,Q}(s,t)(A)-(\wbr{\blt{A}}(t)
-\wbr{\blt{A}}(s))\circ\blt{F_D}(A)\right\|,
\end{equation*}
with
\begin{equation*}
\wbr{\mathfrak{L}}^{n,Q}_{ij}(s,t)(A)=
\frac{1}{n}\sum_{l=1}^{A_{ij}^n(nt)}
\delta_{D_{ijl}}(A)
-
\frac{1}{n}\sum_{l=1}^{A_{ij}^n(ns)}
\delta_{D_{ijl}}(A),
\end{equation*}
and
\begin{equation*}
X^n:=\sup_{A'\in\mathcal{C'}}\sup_{0\leq s\leq t\leq T}
\left\|\wbr{\blt{\mathfrak{L}}}^n(s,t)(A')
-(\wbr{\blt{A}}(t)-\wbr{\blt{A}}(s))\circ\blt{F}(A')\right\|,
\end{equation*}
with
\begin{equation*}
\wbr{\mathfrak{L}}^n_{ij}(s,t)(A')=
\frac{1}{n}\sum_{l=1}^{A_{ij}^n(nt)}
\delta_{\left(B_{ijl}, D_{ijl} \right)}(A')
-
\frac{1}{n}\sum_{l=1}^{A_{ij}^n(ns)}
\delta_{\left(B_{ijl}, D_{ijl} \right)}(A').
\end{equation*}

Then, note that the following bounds hold
\begin{equation*}
\begin{split}
\Gamma^{n,1}_{ij}(t)\leq \wbr{\Q}_{ij}^n(t)(A)-\wbr{\Q}_{ij}^n(0)(A+t)
\leq
\Gamma^{n,2}_{ij}(t)
\end{split}
\end{equation*}
and
\begin{equation*}
\begin{split}
\Gamma^{n,3}_{ij}(t)
\leq \wbr{\Z}_{ij}^n(t)(A')-\wbr{\Z}_{ij}^n(0)(A'+(S_{ij}(\blt{Z}^n,0,t),t))\leq
\Gamma^{n,4}_{ij}(t).
  \end{split}
\end{equation*}
By \cite[Lemma~5.1]{gromoll2008} we have that
\begin{equation*}
\widetilde{X}^n\overset{d} \rightarrow 0\  \text{ and }
X^n\overset{d} \rightarrow 0,
\end{equation*}
as $n \rightarrow \infty$.
By Skorokhod's representation theorem \cite{billingsley1999convergence}, we can assume that all the random elements are defined on a common probability space. Furthermore, by the dominated convergence theorem \cite{rudin1987real}, we have that
$S_{ij}(\wbr{\blt{Z}}^n,s,t)
\rightarrow S_{ij}(\wbr{\blt{Z}},s,t)$, for $s\in[t_0,t]$
as $n\rightarrow \infty$. Moreover, the function $S_{ij}(\wbr{\blt{Z}},s,t)$ is continuous and $S_{ij}(\wbr{\blt{Z}}^n,s,t)$ is monotone in $s$. Hence, we have that
\begin{equation*}
\sup_{t_0 \leq s \leq t}
\left| S_{ij}(\wbr{\blt{Z}}^n,s,t)-S_{ij}
(\wbr{\blt{Z}},s,t)\right|\rightarrow 0.
\end{equation*}
Now, the convergence follows by adapting the conclusion of the proof of \cite[Theorem~5, Section~7.6]{remerova14}.

In the sequel, we show that the fluid limit also satisfies the additional relations in Definition~\ref{def:FMSC}. Observe that by \eqref{eq:Rejected} and the definition of the Riemann-Stieltjes integral, we have that
\begin{equation*}
\wbr{R}^n_{i}(t):=\sum_{j=1}^{J}
\wbr{R}^n_{ij}(t)=\int_{0}^{t}\ind{\wbr{Q}^n_{i}(s^-)=K_i}
d\sum_{j=1}^{J}\wbr{E}^n_{ij}(s).
\end{equation*}
Now, define
\begin{equation*}
\wbr{H}^n_i(t):=\int_{0}^{t} \sum_{j=1}^{J}\lambda_{ij}(s)
\ind{\wbr{Q}^n_{i}(s^-)=K_i} ds,
\end{equation*}
and notice that
\begin{equation*}
\wbr{R}^n_{i}(t)-\wbr{H}^n_{i}(t)=
\int_{0}^{t}\ind{\wbr{Q}^n_{i}(s^-)=K_i}
d\sum_{j=1}^{J}\left(
\wbr{E}^n_{ij}(s)-\int_{0}^{s} \lambda_{ij}(u)du \right).
\end{equation*}
By our assumptions for the arrival process, we obtain that $\wbr{R}^n_{i}(\cdot)-\wbr{H}^n_{i}(\cdot)\overset{d}\rightarrow 0$  as $n\rightarrow \infty$, and hence
$\wbr{H}^n_{i}(\cdot)\overset{d}\rightarrow \wbr{R}_{i}(\cdot)$. Now, by
\eqref{eq:Rejected}, the number of rejected type-$j$ EVs at node $i$ can be written as follows
\begin{equation*}
\wbr{R}^n_{ij}(t)
=\int_{0}^{t}\ind{\wbr{Q}^n_{i}(s^-)=K_i}
d\left(\wbr{E}^n_{ij}(s)-\int_{0}^{s}\lambda_{ij}(u)du\right)
+
\int_{0}^{t}\frac{\lambda_{ij}(s)}{\sum_{h=1}^{J}\lambda_{ih}(s)}
d\wbr{H}^n_{i}(s).
\end{equation*}
Using the assumption of the external arrival process and the fact that $\wbr{H}^n_{i}(\cdot)\overset{d}\rightarrow \wbr{R}_{i}(\cdot)$,
we derive that $\wbr{R}^n_{ij}(\cdot)\overset{d}\rightarrow \wbr{R}_{ij}(\cdot)$ and
\begin{equation*}
\wbr{R}_{ij}(t)=
\int_{0}^{t}
\frac{\lambda_{ij}(s)}{\sum_{h=1}^{J}\lambda_{ih}(s)}
d\wbr{R}_{i}(s).
\end{equation*}

We have proved that any
subsequential limit
$(\wbr{\blt{\Q}}(\cdot),\wbr{\blt{Q}}(\cdot),
\wbr{\blt{\Z}}(\cdot),\wbr{\blt{Z}}(\cdot),\wbr{\blt{R}}(\cdot))$
satisfies the fluid model equations given in  Definition~\ref{def:FMSC}, and hence the proof of
Theorem~\ref{Ch5thm:fluid limit} is completed.

\section{Proofs for Section~\ref{Ch5: Convergence to IP}}
\label{ch5:proofs conv to IP}
\begin{proof}[Proof of Proposition~\ref{Ch5prop:invpointQ}]
First assume that $(\blt{\Q}^*, \blt{q}^*)$ is invariant, i.e.,
$Q_{ij}(t)=q_{ij}^*$ for any $i,j\geq 1$ and $t\geq 0$. We distinguish two cases i) $\sum_{j=1}^{J}q_{ij}^*=K_i$ and ii) $\sum_{j=1}^{J}q_{ij}^*<K_i$.

First, assume that $\sum_{j=1}^{J}q_{ij}^*=K_i$.
By \eqref{eq:Departures} and \eqref{eq:depE}, we have that
\begin{equation}\label{ch5eq:relR1}
\wbr{R}_{ij}(t)=(\lambda_{ij}-\delta_{ij})t.
\end{equation}
Replacing \eqref{ch5eq:relR1} in \eqref{eq:qnet} and taking the limit as time goes to infinity, we obtain
\begin{equation}\label{ch5eq:rel2}
q_{ij}^*=\delta_{ij}\E{D_{ij}}.
\end{equation}
Taking the summation over $j$ in \eqref{ch5eq:relR1} yields
\begin{equation*}
\wbr{R}_{i}(t)=(\sum_{j=1}^{J}
\lambda_{ij}-\sum_{j=1}^{J}\delta_{ij})t.
\end{equation*}
Using the relation
$\wbr{R}_{ij}(t)=\frac{\lambda_{ij}}
{\sum_{h=1}^{J}\lambda_{ih}}\wbr{R}_{i}(t)$
and \eqref{ch5eq:relR1} for $t>0$, we have that
\begin{equation*}
\delta_{ij}=\frac{\lambda_{ij}}
{\sum_{h=1}^{J}\lambda_{ih}}\sum_{j=1}^{J}\delta_{ij}.
\end{equation*}
By \eqref{ch5eq:rel2}, we derive
$q_{ij}^*=\frac{\rho_{ij}}
{\sum_{h=1}^{J}\lambda_{ih}
}\sum_{j=1}^{J}\delta_{ij}$ which yields
$\sum_{j=1}^{J}\delta_{ij}=\frac{\sum_{j=1}^{J}\lambda_{ij}}{\rho_{i}}K_i$ and $\delta_{ij}=\frac{\lambda_{ij}}{\rho_i}K_i$. By \eqref{ch5eq:relR1}, we have that
\begin{equation*}
\wbr{R}_{ij}(t)=\frac{\lambda_{ij}}{\rho_i}
(\rho_i-K_i)t.
\end{equation*}
Replacing the last equation in \eqref{ch5eq:FLQ} and taking $t\rightarrow \infty$, we have that for any Borel set $A\in \mathcal{B}(\R_+)$,
\begin{equation*}
\Q_{ij}^*(A)= \frac{\lambda_{ij}}{\rho_{i}}K_i
 \int_{0}^{\infty} \Prob{D_{ij}\in A+s}ds.
\end{equation*}
Last, by the nonnegativity of $\wbr{R}_{ij}(\cdot)$, we obtain that
$\rho_i> K_i$ in this case.

In the second case, $\sum_{j=1}^{J}q_{ij}^*<K_i$ and by the fluid model equations $\wbr{R}_{ij}(t)=0$ for $t\geq 0$ and for any $i,j\geq 1$. Taking the limit as $t\rightarrow \infty$ in  \eqref{ch5eq:FLQ}, we have that for any Borel set $A\in \mathcal{B}(\R_+)$,
\begin{equation*}
\Q_{ij}^*(A)= \lambda_{ij}
 \int_{0}^{\infty} \Prob{D_{ij}\in A+s}ds.
\end{equation*}

The ``only if'' part of the proposition follows using the same arguments as in the last part of proof in \cite[Theorem~3.6]{kang2015}.
\end{proof}

\begin{proof}[Proof of Proposition~\ref{Ch5prop:invpointZ}]
Uniqueness of the solution of the fixed-point equation follows by \cite[Theorem~2]{remerova14} and \cite{aveklourisstochastic}.
In order to study the asymptotic behavior of $\blt{Z}(\cdot)$, let  $\blt{Q}(0)=\blt{q}^*$. By Proposition~\ref{Ch5prop:invpointQ}, we derive that
$\wbr{R}_{ij}(t)=(\lambda_{ij}-\frac{\lambda_{ij}}{\rho_i}K_i)t$ if $\rho_i>K_i$ and $\wbr{R}_{ij}(t)=0$ if $\rho_i\leq K_i$. This yields $\lambda_{ij}t-\wbr{R}_{ij}(t)=\frac{\lambda_{ij}}{\rho_i}(\rho_i\Min K_i)t$. Now, by the definition of the fluid model (Definition~\ref{def:FMSC}) we have that
\begin{equation*}
\begin{split}
\wbr Z_{ij}(t)=
\wbr Z_{ij}(0) &\Prob{B_{ij}^0 \geq S_{ij}(\blt{z},0,t)  ,D_{ij}^0 \geq t }\\
&+
\frac{\lambda_{ij}}{\rho_i}(\rho_i\Min K_i)\int_{0}^{t}  \Prob{B_{ij} \geq S_{ij}(\blt{z},s,t)  ,D_{ij} \geq t-s } ds.
\end{split}
\end{equation*}
By \cite[Theorem~3]{remerova14}, there exist $\blt{b}^u, \blt{b}^l \in (0,\infty)^{I\times J}$ such that
\begin{equation}\label{eq:b1}
0< b^l_{ij} \leq \liminf_{t\rightarrow \infty} \wbr Z_{ij}(t)\leq
\limsup_{t\rightarrow \infty}\wbr Z_{ij}(t) \leq b^u_{ij}.
\end{equation}
Furthermore, $\blt{b}^u, \blt{b}^l$ satisfy the following relations
\begin{equation}\label{eq:b2}
 b^l_{ij}= \frac{\lambda_{ij}}{\rho_i}(\rho_i\Min K_i)
 \E{D_{ij}\Min \frac{B_{ij}}{
 \sup\limits_{\blt{b}^l\leq \blt{z} \leq \blt{b}^u}
 p_{ij}(\blt{z})}}
\end{equation}
and
\begin{equation*}
 b^u_{ij}= \frac{\lambda_{ij}}{\rho_i}(\rho_i\Min K_i)
 \E{D_{ij}\Min \frac{B_{ij}}{
 \inf\limits_{\blt{b}^l\leq \blt{z} \leq \blt{b}^u}
 p_{ij}(\blt{z})}}.
\end{equation*}
Now, we have assumed that the network is monotone, and hence
$\sup\limits_{\blt{b}^l\leq \blt{z} \leq \blt{b}^u}p_{ij}(\blt{z})=
p_{ij}(\blt{b}^l)$
and
$\inf\limits_{\blt{b}^l\leq \blt{z} \leq \blt{b}^u}p_{ij}(\blt{z})=
p_{ij}(\blt{b}^u)$. Applying the last relation in \eqref{eq:b1} and \eqref{eq:b2}, we have that
\begin{equation*}
 b^l_{ij}= \frac{\lambda_{ij}}{\rho_i}(\rho_i\Min K_i)
 \E{D_{ij}\Min \frac{B_{ij}}{
 p_{ij}(\blt{b}^l)
 }}
\end{equation*}
and
\begin{equation*}
 b^u_{ij}= \frac{\lambda_{ij}}{\rho_i}(\rho_i\Min K_i)
 \E{D_{ij}\Min \frac{B_{ij}}{
p_{ij}(\blt{b}^u)
 }}.
\end{equation*}
In other words, $\blt{b}^u, \blt{b}^l$ satisfy the fixed-point equation~\eqref{ch5eq:invariant} and by the uniqueness of the solution of the fixed-point equation, we obtain
$\blt{b}^u=\blt{b}^l=\blt{z}^*$ and hence
$\lim\limits_{t\rightarrow \infty} \wbr Z_{ij}(t)=z_{ij}^*$.
\end{proof}

\addcontentsline{toc}{section}{Acknowledgements}
\section*{Acknowledgements}
The research of
Angelos Aveklouris is funded by a TOP grant of the Netherlands Organization
for
Scientific Research (NWO) through project 613.001.301. The research of Maria
Vlasiou is supported by the NWO MEERVOUD grant 632.003.002.
The research of Bert Zwart is partly supported by the NWO VICI grant
639.033.413.

\phantomsection
\addcontentsline{toc}{section}{References}
\begin{footnotesize}
\bibliography{Mybibliography_Angelos_Ab}
\bibliographystyle{abbrv}
\end{footnotesize}
\end{document}